\documentclass{article}
\usepackage[margin= 1in]{geometry}
\usepackage{times}
\usepackage{graphicx} 
\usepackage{subfigure} 
\usepackage{epsfig}
\usepackage{natbib}

\usepackage{algorithm}
\usepackage{algorithmic}

\usepackage{hyperref}

\usepackage{amsmath}
\usepackage{amsthm}
\usepackage{amssymb}
\usepackage{amsfonts}
\usepackage{mathtools,xparse}
\usepackage{epstopdf}
\usepackage{multirow}

\newcommand{\mb}[1]{\mathbf{#1}}
\newcommand{\mc}[1]{\mathcal{#1}}
\newcommand{\mr}[1]{\mathbb{#1}}

\DeclarePairedDelimiter{\norm}{\lVert}{\rVert}
\NewDocumentCommand{\normL}{ s O{} m }{%
  \IfBooleanTF{#1}{\norm*{#3}}{\norm[#2]{#3}}_{L_2(\Omega)}%
}
\def\k{{(k)}}
\newtheorem{Lemma}{Lemma}

\newtheorem{Corollary}{Corollary}
\newtheorem{Theorem}{Theorem}
\title{Semi-Stochastic Frank-Wolfe Algorithms with Away-Steps for Block-Coordinate Structure Problems}
\author{Donald Goldfarb, Garud Iyengar, and Chaoxu Zhou}

\begin{document}
\maketitle
\begin{abstract} 
We propose a semi-stochastic Frank-Wolfe algorithm with away-steps for regularized empirical risk minimization and extend it to problems with block-coordinate structure. Our algorithms use adaptive step-size and we show that they converge linearly in expectation. The proposed algorithms can be applied to many important problems in statistics and machine learning including regularized generalized linear models, support vector machines and many others. In preliminary numerical tests on structural SVM and graph-guided fused LASSO, our algorithms outperform other competing algorithms in both iteration cost and total number of data passes.
\end{abstract} 

\section{Introduction}
\label{intro}
\subsection{Motivations}
The recent trend of using a large number of parameters to model large datasets in machine learning and statistics has created a strong demand for optimization algorithms that have low computational cost per iteration and exploit model structure. Empirical risk minimization is a class of important optimization problems in this area. Such problems can be written as
\begin{align}
\min_{\mb{x} \in \mc{P}} F(\mb{x}) \equiv \frac{1}{n}\sum_{i = 1}^n f_i(\mb{x}), \label{erm_problem}
\end{align}
where $\mc{P}$ is a compact polyhedral set in $\mr{R}^p$ and each $f_i(\cdot)$ is a convex function. A popular approach for solving (\ref{erm_problem}) is the proximal gradient method which solves a projection sub-problem in each iteration. The major drawback of this method is that the projection step can be expensive in many situations. As an alternative, the Frank-Wolfe (FW) algorithm \cite{FW56}, also known as the conditional gradient method, solves a linear optimization sub-problem in each iteration which is much faster than the standard projection technique when the feasible set is a simple polytope \cite{N15}. On the one hand, when the number of observations $N$ in problem (\ref{erm_problem}) is large, calculating the gradient of $F(\cdot)$ in every FW iteration becomes a computationally intensive task. The question of whether `cheap' stochastic gradient can be used as a surrogate in FW immediately arises. On the other hand, when the dimension of the parameter space $d$ in problem (\ref{erm_problem}) is large, taking advantage of any underlying structure in the problem can further accelerate the algorithm. In this paper, we present a linearly convergent semi-stochastic Frank-Wolfe algorithm with away-steps for empirical risk minimization and extend it to problems with block-coordinate structure. Specifically, we apply the proposed algorithm to structural support vector machine and graph-guided fused lasso problems.
\subsection{Contributions}
The contribution of this paper has two aspects. On the theoretical side, it presents the first stochastic conditional gradient algorithm that converges linearly in expectation. In addition, this algorithm uses an adaptive step size rather than one determined by exact line search. On the application side, the algorithm can be applied to various statistical and machine learning problems including support vector machine, multi-task learning, regularized generalized linear models and many others. Furthermore, when the problem has a block-coordinate structure, performance of this algorithm can be enhanced by extending it to a randomized coordinate version which solves several small sub-problems instead of a large problem.
\subsection{Related Work}
The Frank-Wolfe algorithm was proposed sixty years ago  \cite{FW56} for minimizing a convex function over a polytope and is known to converge at an $O(1/k)$ rate. In \cite{LP66} the same convergence rate was proved for compact convex constraints. When both objective function and the constraint set are strongly convex, \cite{GH15} proved that the Frank-Wolfe algorithm has an  $O(1/k^2)$ rate of convergence with a properly chosen step size.  Motivated by removing the influence of ``bad" visited vertices, the away-steps variant of the Frank-Wolfe algorithm was proposed in \cite{Wol70}. Later, \cite{GM86} showed that this variant converges linearly under the assumption that the objective function is strongly convex and the optimum lies in the interior of the constraint  polytope. Recently, \citep{GH13} and \cite{LJJ13} extended the linear convergence result by removing the assumption of the location of the optimum and \cite{BS15} extended it further by relaxing the strongly convex objective function assumption. Stochastic Frank-Wolfe algorithms have been considered by \cite{LAN13} and \cite{LWM15} in which an $O(1/k)$ rate of convergence in expectation is proved. In addition, the Frank-Wolfe algorithm has been applied to solve several different classes of problems, including non-linear SVM \cite{OG10}, structural SVM \cite{LJJSP13}, and comprehensive principal component pursuit \cite{MZWG15} among many others.
\subsection{Notation}
Bold letters are used to denote vectors and matrices, normal fonts are used to denote scalers and sets. Subscripts represent elements of a vector,  while superscripts with parentheses represent iterates of the vector, i.e. $\mathbf{x}^{(k)}$ is a vector at iteration $k$, and $x_i^{(k)}$ is the $i$-th element of $\mathbf{x}^{(k)}$. The cardinality of the set $I$ is denoted by $\vert I \vert$. Given two vectors $\mathbf{x}, \mathbf{y} \in \mathbb{R}^n$, their inner product is denoted by $\langle \mathbf{x}, \mathbf{y} \rangle$. Given a matrix $\mathbf{A}\in \mathbf{R}^{m \times n}$ and vector $\mathbf{x} \in \mathbb{R}^n$, $\norm{\mathbf{A}}$ denotes the spectral norm of $\mathbf{A}$, and $\norm{\mathbf{x}}$ denotes the $l_2$ norm of $\mathbf{x}$. $\mathbf{A}^\top$ represent the transpose of $\mathbf{A}$. We denote the $i$th row of a given matrix $\mathbf{A}$ by $\mathbf{A}_i$, and for a given set $I \subset \{1, \ldots, m\}$, $\mathbf{A}_I \in \mathbf{R}^{\vert I \vert \times n}$is the submatrix of $\mathbf{A}$ such that $(\mathbf{A}_I)_j = \mathbf{A}_{I_j}$ for any $j = 1, \ldots, \vert I \vert$. Given matrices $\mathbf{A} \in \mathbb{R}^{n \times m}$ and $\mathbf{B} \in \mathbb{R}^{n \times k}$, the matrix $[\mathbf{A}, \mathbf{B}] \in \mathbb{R}^{n \times (m + k)}$ is their horizontal concatenation. Let $\lceil x \rceil$ be the ceiling function that rounds $x$ to the smallest integer larger than $x$.

\section{Stochastic Condtional Gradient Method for Empirical Risk Minimization}
\subsection{Problem description, notation and assumptions}
Consider the minimization problem
\begin{align}
\min_{\mb{x} \in \mc{P}} \Big\{F(\mb{x}) \equiv \frac{1}{n}\sum_{i = 1}^n f_i(\mb{a}^\top_i \mb{x}) + \langle \mb{b}, \mb{x} \rangle \Big\}, \label{general_problem} \tag{P1}
\end{align}
where $\mathcal{P}$ is a non-empty compact polyhedron given by $\mathcal{P} = \{x \in \mathbb{R}^p : \mathbf{Cx} \leq \mb{d}\}$ for some $\mathbf{C} \in \mathbb{R}^{m \times p}$, $\mathbf{d} \in \mathbb{R}^m$. For every $i = 1, \ldots, n$, $\mb{a}_i \in \mr{R}^p$, and $f_i : \mr{R} \rightarrow \mr{R}$ is a strongly convex function with parameter $\sigma_i$ and has a Lipschitz continuous gradient with constant $L_i$. $\mb{b}$ in the linear term of (\ref{general_problem}) is a vector in $\mr{R}^p$. Note that the gradient of $F(\cdot)$ is also Lipschitz continuous with constant $L \leq (\sum_{i=1}^nL_i \norm{a_i}) / n$. However, because of the affine transformation in the argument of each $f_i(\cdot)$, $F(\cdot)$ may not be a strongly convex function.

\noindent\textbf{Remark:}
\begin{description}
\item[1] Many statistics and machine learning problems can be modeled as problem (\ref{general_problem}). For example:
\begin{description}
\item[i] The LASSO problem : $\min_{\boldsymbol\beta \in \mathbb{R}^p} \sum_{i=1}^n (y_i - \mb{x}_i^\top\boldsymbol\beta)^2 + \lambda \vert\vert \boldsymbol\beta \vert\vert_1$, where $n$ is the sample size and for $i = 1, 2,\ldots, n$, $y_i$'s are responses and $\mb{x}_i$'s are the covariates. $\boldsymbol\beta$ is the regression coefficient to be estimated by solving the minimization problem and $\lambda$ is a regularization parameter for sparsity.
 
\item[ii]$l_1$-Regularized Poisson Regression:  $\min_{\boldsymbol\beta \in \mathbb{R}^p} \sum_{i = 1}^n \exp\{\mathbf{x}_i^\top\boldsymbol\beta\} - y_i\mathbf{x}_i^\top\boldsymbol\beta +\log(y_i!) +  \lambda \vert\vert \boldsymbol\beta \vert\vert_1 $, where $n$ is the sample size and for $i = 1, 2, \ldots, n$, $\mb{x}_i$'s are the covariates and $y_i \in \mr{N}$ are the responses; that is, $y_i$'s are obtained from event counting. $\boldsymbol\beta$ is the regression coefficient to be estimated by solving the minimization problem and $\lambda$ is a regularization parameter for sparsity.

\item[iii] $l_1$-Regularized Logistic Regression: $\min_{\boldsymbol\beta \in \mathbb{R}^p}\sum_{i = 1}^n \log(1 + \exp\{-y_i\mathbf{x}_i^\top\boldsymbol\beta\})  + \lambda \vert\vert \boldsymbol\beta \vert\vert_1$, where $n$ is the sample size and for $i = 1, 2, \ldots, n$, $\mb{x}_i$'s are the covariates and $y_i \in \{0, 1\}$ are the responses; that is, $y_i$'s are binary labels obtained from classification. $\boldsymbol\beta$ is the regression coefficient to be estimated by solving the minimization problem and $\lambda$ is a regularization parameter for sparsity.

\item[iv] Dual Problem of $l_1$-loss Support Vector Machine: given training label-instance pairs $(y_i, \mb{z}_i)$ for $i = 1, \ldots, l$, the problem is formulated as 
\begin{align*}
&\text{minimize}_{\boldsymbol\alpha}\quad  \frac{1}{2}\boldsymbol\omega^\top\boldsymbol\omega - \boldsymbol1^\top\boldsymbol\alpha\\
&\text{subject to}\quad  \boldsymbol\omega = \mathbf{A}\boldsymbol\alpha, \\
&\quad \quad \quad \quad \quad \; 0 \leq \alpha_i \leq C, i = 1, \ldots, l.
\end{align*}
where $\mb{A} = [y_1\mb{z}_1, \ldots, y_l\mb{z}_l]$, $\boldsymbol1$ is the vector of ones, and $C$ is a given upper bound. This problem can be transformed to the form of (\ref{general_problem}) by replacing the $\boldsymbol\omega$ is in the objective function by $\mb{A}\boldsymbol\alpha$.
\end{description}
\item[2] The objective functions in the unconstrained problems (i),(ii) and (iii) all involve a non-smooth regularization term. They however, can be modeled as (\ref{general_problem}). For example, in the LASSO problem, we can always take $\boldsymbol\beta = \boldsymbol0$ to get an upper bound on the function value and add the constraint $\norm{\boldsymbol\beta}_1 \leq \sum_{i=1}^n y_i^2$ to the original problem. Then by introducing new variables $u_i$, $i = 1, \ldots, p$, we can express $\norm{\boldsymbol\beta}_1$ as $\sum_{i=1}^p u_i$ after adding the constraints $\beta_i \leq u_i$ and $-\beta_i \leq u_i$ for all $i = 1, \ldots, p$ to the problem. We can apply the same method to $l_1$-regularized logistic regression and $l_1$-regularized Poisson regression problems.
\end{description}

\subsection{The Main Result}
Let $\mathcal{O}: \mathbb{R}^p \rightarrow \mathcal{P}$ be a linear oracle that given $\mathbf{c} \in \mathbb{R}^p$, returns $\mathbf{z} = \mathcal{O}(\mathbf{c}) \in \mathcal{P}$ that $\mathbf{c}^\top\mb{z} \leq \mathbf{c}^\top \mb{x}$ for every $\mb{x} \in \mc{P}$; i.e. $\mb{z} = \arg\min\{\mb{c}^\top \mb{x} \; \vert \; \mb{x} \in \mc{P}\}$. Let $V$ be the set of vertices of polytope $\mc{P}$.

\begin{algorithm}
\caption{Semi-Stochastic Frank-Wolfe algorithm with Away-Steps}
\label{cond_grad_1}
\begin{algorithmic}[1]
\STATE {\bfseries Input:}  $\mathbf{x}^{(1)} \in V$, $f_i$, $\mb{a}_i$, $\mb{b}$ and $L$
\STATE Set $\mu^{(1)}_{\mathbf{x}^{(1)}} = 1$, $\mu_\mathbf{v}^{(1)} = 0$ for any $\mathbf{v} \in \mathcal{V} / \{\mathbf{x}^{(1)}\}$ and $U^{(1)} = \{\mathbf{x}^{(1)}\}$.
\FOR{$k = 1, 2, \ldots$}
\STATE Uniformly sample $\mc{J} =\{j_1, \ldots, j_{m_k}\}$ from $\{1, \ldots, n\}$ without replacement, and denote  $\mathbf{g}^{(k)} = \frac{1}{m_k}\sum_{i=1}^{m_k}f'_{j_i}(\mb{a}_{j_i}^\top\mb{x}^{(k)})\mb{a}_{j_i} + \mb{b}$.
\STATE Compute $\mathbf{p}^{(k)} = \mathcal{O}(\mathbf{g}^{(k)})$.
\STATE Compute $\mathbf{u}^{(k)} \in {\arg\!\max}_{\mathbf{v} \in U^{(k)}}\langle \mathbf{g}^{(k)}, \mathbf{v} \rangle$.
\IF{$\langle \mathbf{g}^{(k)}, \mathbf{p}^{(k)}  + \mathbf{u}^{(k)}- 2\mathbf{x}^{(k)} \rangle \leq 0$}
\STATE Set $\mathbf{d}^{(k)} = \mathbf{p}^{(k)} - \mathbf{x}^{(k)}$ and $\gamma^{(k)}_{\text{max}} = 1$. 
\ELSE
\STATE Set $\mathbf{d}^{(k)} = \mathbf{x}^{(k)} - \mathbf{u}^{(k)}$ and $\gamma^{(k)}_{\text{max}} = \frac{\mu^{(k)}_{\mathbf{u}^{(k)}}}{1 - \mu^{(k)}_{\mathbf{u}^{(k)}}}$.
\ENDIF
\STATE Set $\gamma^{(k)} = \min\{-\frac{\langle \mb{g}^{(k)}, \mb{d}^{(k)}\rangle}{L\norm{\mb{d}^{(k)}}^2}, \gamma^{(k)}_\text{max}\}$
\STATE Set $\mathbf{x}^{(k+1)} = \mathbf{x}^{(k)} + \gamma^{(k)}\mathbf{d}^{(k)}$.
\STATE Update $U^{(k + 1)}$ and $\mathbf{\mu}^{(k+1)}$ by Procedure VRU.
\ENDFOR
\STATE{\bfseries Return:} $\mathbf{x}^{(k+1)}$.
\end{algorithmic}
\end{algorithm}
\noindent The following algorithm updates a vertex representation of the current iterate and is called in Algorithm \ref{cond_grad_1}.
\begin{center}
\begin{algorithm}[h]
\caption{Procedure Vertex Representation Update (VRU)}
\begin{algorithmic}[1]
\STATE {\bfseries Input:} $\mb{x}^\k$, $(U^\k, \boldsymbol\mu^\k)$, $\mb{d}^\k$, $\gamma^\k$, $\mb{p}^\k$ and $\mb{v}^\k$.
\IF{$\mb{d}^\k = \mb{x}^\k - \mb{u}^\k$}
\STATE Update $\mu^\k_\mb{v} = \mu_\mb{v}^\k(1 + \gamma^\k)$ for any $\mb{v} \in U^\k / \{\mb{u}^\k\}$.
\STATE Update $\mu^{(k+1)}_{\mb{u}^\k} = \mu^\k_{\mb{u}^\k}(1 + \gamma^\k) - \gamma^\k$.
\IF{$\mu^{(k+1)}_{\mb{u}^\k} = 0$}
\STATE Update $U^{(k+1)} = U^\k /\{\mb{u}^\k\}$
\ELSE
\STATE Update $U^{(k+1)} = U^\k$
\ENDIF
\ENDIF
\STATE Update $\mu_\mb{v}^{(k+1)} = \mu^\k_\mb{v}(1 - \gamma^\k)$ for any $\mb{v} \in U^\k / \{\mb{p}^\k\}$.
\STATE Update $\mu^{(k+1)}_{\mb{p}^\k} = \mu^\k_{\mb{p}^\k}(1 - \gamma^\k) + \gamma^\k$.
\IF{$\mu_{\mb{p}^\k}^{(k+1)} = 1$}
\STATE Update $U^{(k+1)} = \{\mb{p}^\k\}$.
\ELSE
\STATE Update $U^{(k+1)} = U^\k \cup \{\mb{p}^\k\}$.
\ENDIF
\STATE (Optional) Carath\'eodory's theorem can be applied for the vertex representation of $\mb{x}^{(k+1)}$ so that $\vert U^{(k+1)}\vert = p+1$ and $\boldsymbol\mu^{(k+1)} \in \mr{R}^{p+1}$.
\STATE {\bfseries Return:} $(U^{(k+1)}, \boldsymbol\mu^{(k+1)})$
\end{algorithmic}
\end{algorithm}
\end{center}

\noindent We need to introduce some definitions before presenting the theorems in this paper. Write $\mb{A} = [\mb{a}_1, \mb{a}_2, \ldots, \mb{a}_n]^\top$ and $\mb{G}(\mb{x})= (1/n)[f_1'(\mb{a}_1^\top\mb{x}), \ldots, f_n'(\mb{a}_n^\top\mb{x})]^\top $. Define $D = \sup\{\norm{\mb{x} - \mb{y}}\, \vert \, \mb{x}, \mb{y} \in \mc{P}\}$ and $G = \sup\{\norm{\mb{G}(\mb{x})}\, \vert\, \mb{x} \in \mc{P}\}$. It follows from compactness of $\mc{P}$ and continuity of $F(\cdot)$ that $D < \infty$ and $G < \infty$. Write $\kappa =\theta^2 \{D\norm{\mb{b}} + 3GD\norm{\mb{A}} + \frac{2n}{\sigma_F}(G^2+1)\}$ where $\sigma_F = \min\{\sigma_1, \ldots, \sigma_n\} > 0$ and $\theta$ is the Hoffman constant associated with matrix $[\mb{C}^\top,\mb{A}^\top, \mb{b}^\top]$ that is $\theta = \max\{1 / \lambda_{\min}(\mb{BB}^\top) \;\vert \; \mb{B} \in B\}$, $\lambda_{\min}(\mb{BB}^\top)$ is the smallest eigenvalue of $\mb{BB}^\top$, where $B$ is the set of all matrices constructed by taking linearly independent rows from the matrix $[\mb{C}^\top,\mb{A}^\top, \mb{b}^\top]$. We denote by $I(\mb{x})$, the index set of active constraints at $\mb{x}$; that is,  $I(\mb{x}) = \{i \in {1, \ldots, m} \; \vert \; \mb{C}_i\mb{x} = d_i\}$. In a similar way, we define the set of active constraints for a set $U$ by  $I(U) = \{i \in \{1, \ldots, m\} \; \vert \; \mb{C}_i\mb{v} = d_i, \forall \mb{v} \in U\} = \cap_{\mb{v} \in U}I(\mb{v})$. Let $V$ be the set of vertices of $\mc{P}$, then define $\Omega_\mathcal{P} = \frac{\zeta}{\phi}$ where 
\begin{align*}
\zeta &= \min_{\mathbf{v} \in V, i \in \{1, \ldots, m\}: d_i > \mathbf{C}_i \mathbf{v}} (d_i - \mathbf{C}_i\mathbf{v}), \\
\phi &= \max_{i \in \{1, \ldots, m\} / I(V)}\norm{\mathbf{C}_i}.
\end{align*}
\noindent\textbf{Remark:} The vertex representation update procedure can also be implemented by using Carath\'eodory's theorem so that each $\mb{x}^\k$ can be written as a convex combination of at most $N = p+1$ vertices of the polytope $\mc{P}$. Then the set of vertices $U^{(k)}$ and their corresponding weights $\boldsymbol\mu^{(k)}$ can be updated according to the convex combination.
\begin{Theorem}
Let $\{\mathbf{x}^{(k)}\}_{k \geq 1}$ be the sequence generated by Algorithm \ref{cond_grad_1} for solving Problem $(\ref{general_problem})$, $N$ be the number of vertices used to represent $\mb{x}^{(k)}$ (if VRU is implemented by using Carath\'eodory's theorem, $ N = p + 1$, otherwise $N = \vert V \vert$) and $F^*$ be the optimal value of the problem. Let $\rho = \Omega_{\mathcal{P}}^2 /\{8N^2 \kappa D\max (G, LD)\}$ and set $m_k = \lceil n / (1 + n(1 - \rho)^{2\alpha k})\rceil$ for some $0 < \alpha < 1$. Let $D^{(k)}$ be the event that the algorithm removes a vertex from the current convex combination at iteration $k$, that is, the algorithm performs a `drop step' at iteration $k$. Assume $\mr{P}(D^{(k)}) \leq (1 - \rho)^{\beta k}$ for some $ 0 < \beta < 1$.
Then for every $k \geq 1$
\begin{align}\label{rate_of_convergence}
\mathbb{E}\{F(\mathbf{x}^{(k+1)}) - F^*\} \leq C_3(1 - \rho)^{\min (\alpha, \beta)k}
\end{align}
where
\begin{align*}
C_3 = F(\mb{x}^{(1)}) - F^*  + \frac{G\sqrt{C_2}}{L\{1 - (1 - \rho)^{1 -\alpha}\}} + \frac{G^2}{2L\{1 - (1 - \rho)^{1 - \beta}\}}
\end{align*}
and 
\begin{align*}
C_2 &= \sup_{\mb{x} \in \mc{P}}[ \frac{1}{n}\sum_{i=1}^n \{f_i'(\mb{a}_i^\top\mb{x})\}^2\mb{a}_i^\top\mb{a}_i \\
&\quad - \frac{1}{n(n-1)}\sum_{i \neq j}f_i'(\mb{a}_i^\top\mb{x})f_j'(\mb{a}_i^\top\mb{x})\mb{a}^\top_i\mb{a}_j ] < \infty.
\end{align*}
\end{Theorem}
\noindent Proof of the Theorem 1 is presented in the appendix. It is worth noting that by dividing both sides of (\ref{rate_of_convergence}) by $F^{(1)} - F^*$, this linear convergence result for Algorithm 1 depends on relative function values. Therefore, the linear convergence of the proposed algorithm is invariant to scaling of the function.

\noindent\textbf{Remarks}: The constant $\rho$ depends on the "vertex-face" distance of a polytope as discussed in \cite{BS15}. This also gives some intuition as to why constraint set $\mc{P}$ has to be a polytope for linear convergence of the Frank-Wolfe algorithm with away-steps. When the boundary of the constraint set is `curved' as is the case for the $l_2$-ball, every point on the boundary is an extreme point. Then, faces can always be constructed so that the vertex-face distance is infinitesimal. Thus, we cannot get linear convergence when we apply the proof to a general convex constraint.

\section{Block Coordinate Semi-Stochastic Frank-Wolfe Algorithm with Away-Steps}
In this section, we assume the domain $\mc{P}$ takes the form $\mc{P} = \mc{P}_{[1]} \times \mc{P}_{[2]} \times \cdots \times \mc{P}_{[q]}$, where each $\mc{P}_{[i]}$ is a compact polytope that can be expressed as $\mc{P}_{[i]} = \{\mb{x} \in \mr{R}^{p_i} \, \vert \, \mb{C}_{[i]}\mb{x} \leq \mb{d}_{[i]}\}$ where $ \mb{C}_{[i]} \in \mr{R}^{m_i \times p_i}$, $\mb{d}_{[i]} \in \mr{R}^{m_i}$, $\sum_{i = 1}^q m_i = m$ and $\sum_{i=1}^q p_i = p$. Hence, $\mc{P}$ still has the polytopic representation $\mc{P} = \{\mb{x} \in \mr{R}^p \: \vert \: \mb{C}\mb{x} \leq \mb{d}\}$ where
\begin{align*}
\mb{C}= \begin{bmatrix}
\mb{C}_{[1]} & & &\\
& \mb{C}_{[2]} & & \\
& &\ddots & \\
& & & \mb{C}_{[q]}
\end{bmatrix} \quad \quad \text{ and }\quad  \quad \mb{d} = \begin{bmatrix}
\mb{d}_{[1]}\\
\mb{d}_{[2]}\\
\vdots\\
\mb{d}_{[q]}
\end{bmatrix}.
\end{align*}
Examples of such Cartesian product constraints include the dual problem of structural SVMs, fitting marginal models in multivariate regression and multi-task learning. Let $V_{[i]}$ be the set of vertices of polytope $\mc{P}_{[i]}$. We use subscripts $[i]$ to denote vectors, matrices, sets and other quantities correspond to the $i$-th block of constraints. Specifically,
let $\mb{x}_{[i]} \in \mr{R}^{p_i}$ and $\mb{x} = [\mb{x}_{[1]}^\top, \mb{x}_{[2]}^\top, \cdots, \mb{x}_{[q]}^\top]^\top \in \mr{R}^p$. Define $\mb{U}_{[i]} \in \mr{R}^{p\times p_i}$ as the sub-matrices of $p \times p$ identity matrix that corresponds to the $i$-th block of the product domain, that is  $[\mb{U}_{[1]},\mb{U}_{[2]}, \ldots, \mb{U}_{[q]}] = \mb{I}_{p \times p}$. Hence $\mb{x}_{[i]} = \mb{U}_{[i]}^\top\mb{x}$ and $\mb{x} = \sum_{i = 1}^q \mb{U}_{[i]}\mb{x}_{[i]}$. Let $\mc{O}_{[i]}$ be the linear oracle corresponding to the $i$-th block polytope in a lower dimension space. With the above notation, we are ready to present the block-coordinate version of the semi-stochastic Frank-Wolfe algorithm with away-steps and prove its linear convergence.

\begin{algorithm}[t]
\caption{Block Coordinate Semi-Stochastic Frank-Wolfe Algorithm with Away-Steps}
\label{block_cond_grad_1}
\begin{algorithmic}[1]
\STATE {\bfseries Input:}  $\mathbf{x}^{(1)} \in V_{[1]} \times \cdots \times V_{[q]}$, $f_i$, $\mb{a}_i$, $\mb{b}$ and $L$
\STATE Let $\mu^{(1)}_{\mathbf{x}^{(1)}_{[i]}} = 1$, $\mu_\mathbf{v}^{(1)} = 0$ for any $\mathbf{v} \in V_{[i]} / \{\mathbf{x}^{(1)}_{[i]}\}$ and $U^{(1)}_{[i]} = \{\mathbf{x}^{(1)}_{[i]}\}$.
\FOR{$k = 1, 2, \ldots$}
\STATE Uniformly sample $\mc{J} =\{j_1, \ldots, j_{m_k}\}$ from $\{1, \ldots, n\}$ without replacement, and denote  $\mathbf{g}^{(k)} = \frac{1}{m_k}\sum_{i=1}^{m_k}f'_{j_i}(\mb{a}_{j_i}^\top\mb{x}^{(k)})\mb{a}_{j_i} + \mb{b}$.
\STATE Uniformly sample $\mc{L} = \{l_1, \ldots, l_r \}$ from $\{1, \ldots, q\}$ without replacement
\FOR{$i = 1, 2, \ldots, r$, in parallel}
	\STATE Compute $\mathbf{p}^{(k)}_{[l_i]} = \mathcal{O}_{[l_i]}(\mb{U}_{[l_i]}^\top\mathbf{g}^{(k)})$
	\STATE Compute $\mathbf{u}^{(k)}_{[l_i]} \in {\arg\!\max}_{\mathbf{v} \in U^{(k)}_{[l_i]}}\langle\mb{U}_{[l_i]}^\top 		\mathbf{g}^{(k)}, \mathbf{v} \rangle $.
\IF{$\langle \mb{U}_{[l_i]}^\top\mathbf{g}^{(k)}, \mathbf{p}^{(k)}_{[l_i]}  + \mathbf{u}^{(k)}_{[l_i]}- 2\mathbf{x}^{(k)}_{[l_i]}\rangle \leq 0$}
\STATE Set $\mathbf{d}^{(k)}_{[l_i]} = \mathbf{p}^{(k)}_{[l_i]} - \mathbf{x}^{(k)}_{[l_i]}$ and $\bar{\gamma}^{(k)}_{[l_i]} = 1$. 
\ELSE
\STATE Set $\mathbf{d}^{(k)}_{[l_i]} = \mathbf{x}^{(k)}_{[l_i]} - \mathbf{u}^{(k)}_{[l_i]}$ and $\bar{\gamma}^{(k)}_{[l_i]} = \frac{\mu^{(k)}_{\mathbf{u}^{(k)}_{[l_i]}}}{1 - \mu^{(k)}_{\mathbf{u}^{(k)}_{[l_i]}}}$.
\ENDIF
\STATE Set $\gamma^{(k)}_{[l_i]} = \min\{-\frac{\langle \mb{U}_{[l_i]}^\top \mb{g}^{(k)}, \mb{d}^{(k)}_{[l_i]} \rangle}{L\norm{\mb{d}^{(k)}_{[l_i]}}^2}, \bar{\gamma}^{(k)}_{[l_i]} \}$.
	\STATE Update $U^{(k + 1)}_{[i]}$ and $\mathbf{\mu}^{(k+1)}_{[i]}$ by Procedure VRU.
\ENDFOR
\STATE Update $\mathbf{x}^{(k+1)} = \mathbf{x}^{(k)} + \sum_{i= 1}^r \gamma^{(k)}_{[l_i]}\mb{U}_{[l_i]}\mathbf{d}^{(k)}_{[l_i]}$.
\ENDFOR
\STATE {\bfseries Return:} $\mathbf{x}^{(k+1)}$
\end{algorithmic}
\end{algorithm}

\begin{Theorem}
Let $\{\mathbf{x}^{(k)}\}_{k \geq 1}$ be the sequence generated by Algorithm \ref{block_cond_grad_1} for solving Problem $(\ref{general_problem})$ with block-coordinate structure constraints ,$N$ be the number of vertices used to represent $\mb{x}^{(k)}$ (if VRU is implemented by using Carath\'eodory's theorem, $ N = p + 1$, otherwise $N = \vert V \vert$) and $F^*$ be the optimal value of the problem. Let $\hat{\rho} = r\Omega_{\mathcal{P}}^2 /\{8N^2 \kappa q^2D \max (G, LD)\}$ and set $m_k = \lceil n / (1 + n(1 - \hat{\rho})^{2\mu k})\rceil$ for some $0 < \mu < 1$. Let $D^{(k)}_{[i]}$ be the event that the algorithm removes a vertex from the current convex combination in block $i$ at iteration $k$, that is, the algorithm performs a `drop step' at iteration $k$. Assume $\max_{i}\{\mr{P}(D^{(k)}_{[i]})\} \leq (1 - \hat{\rho})^{\lambda k}$ for some $ 0 < \lambda < 1$ and all $k \geq 1$. Then for any $k \geq 1$
\begin{align*}
\mathbb{E}\{F(\mathbf{x}^{(k+1)}) - F^*\}
&\leq C_4(1 - \hat\rho)^{\min(\mu, \lambda)}
\end{align*}
where
\begin{align*}
C_4 = \{F(\mb{x}^{(1)}) - F^*  + \frac{rG\sqrt{C_2}}{qL\{1 - (1 - \hat{\rho})^{1 -\mu}\}} + \frac{rG^2}{2L\{1 - (1 - \hat{\rho})^{1 - \lambda}\}}\}.
\end{align*}
\end{Theorem}
\noindent Proof of the Theorem 2 can be found in the appendix.

\noindent\textbf{Remark:} All of the sub-problems in the inner loop of the proposed algorithm can be solved in parallel. This feature can further boost the performance of the algorithm. It seems that the rate of convergence of the algorithm for block-coordinate problems is worse since $\hat{\rho} < \rho$. One reason for this is the algorithm only uses sampled blocks instead of a full data pass in every iteration. Another reason is that the theoretical result is based on the worst case scenario happening at the same time in every sampled block (which won't happen in real implementations) for the ease of proving linear convergence. The worst case scenario is the so called `drop-step' and the detailed analysis can be found in the supplementary material.
\section{Numerical Studies}
In this section, we apply the semi-stochastic Frank-Wolfe algorithm with away-steps to two popular problems in machine learning. The first one is a graph guided fused LASSO problem for multi-task learning using a simulated data set. The second one is a structural support vector machine (SVM) problem using a real data set in speech recognition.
\subsection{Multi-task Learning via Graph Guided Fused LASSO}
\begin{figure*}
\centering
\begin{tabular}{rrl}
\includegraphics[scale = 0.28]{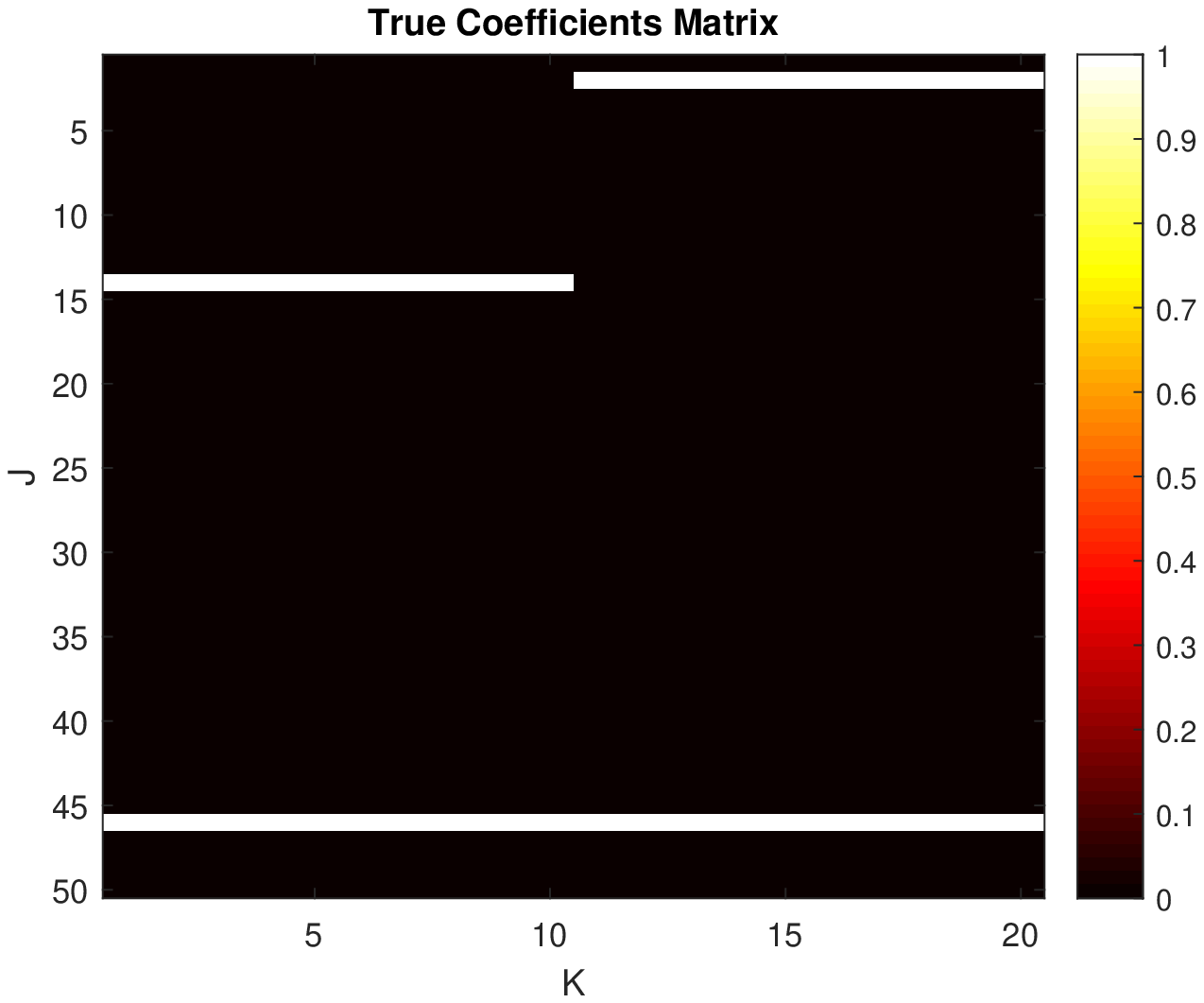}&
\includegraphics[scale = 0.28]{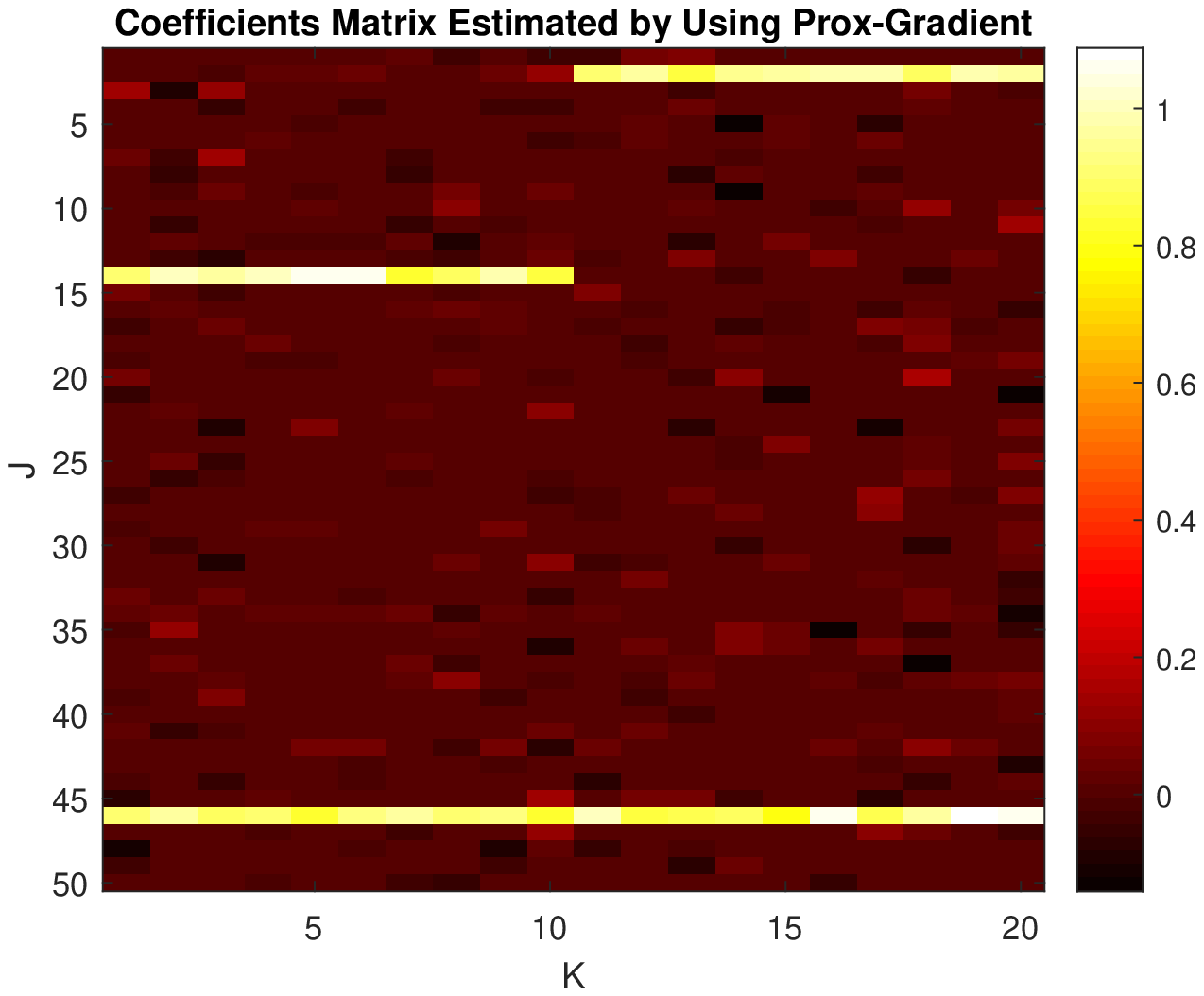} &
\multirow{2}*[2.7cm]{\includegraphics[scale=0.5]{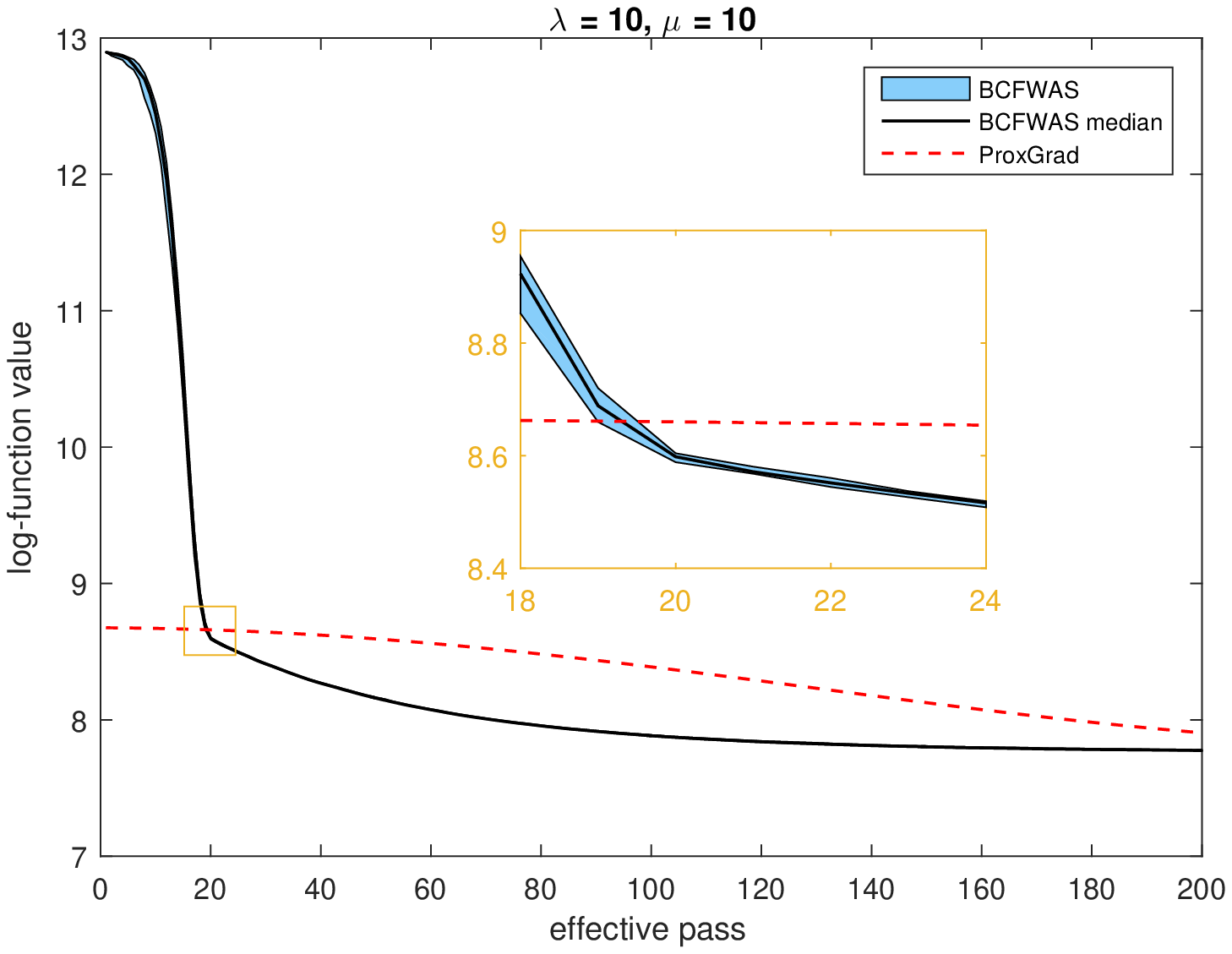}}\\
\includegraphics[scale = 0.28]{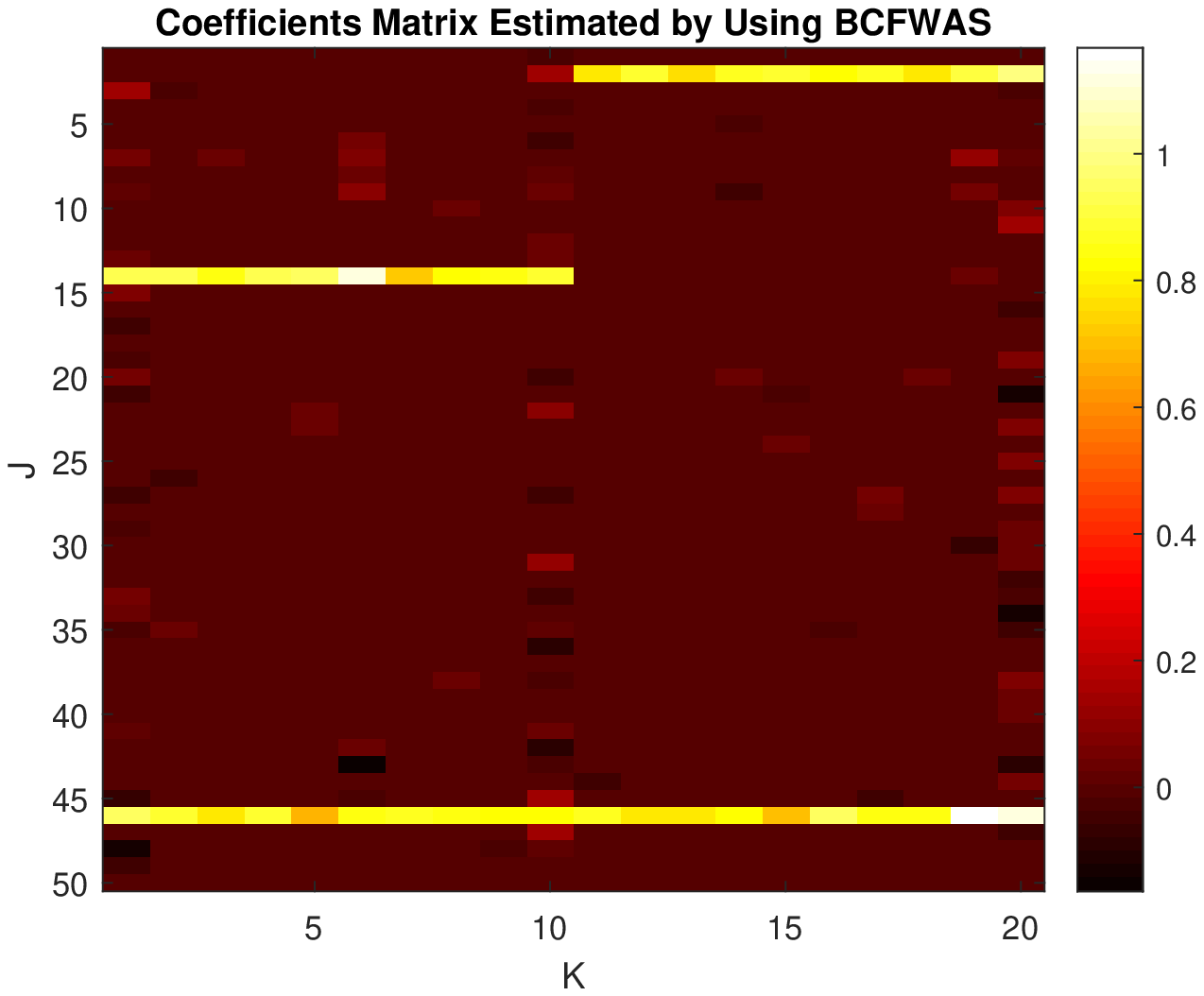} &
\includegraphics[scale = 0.28]{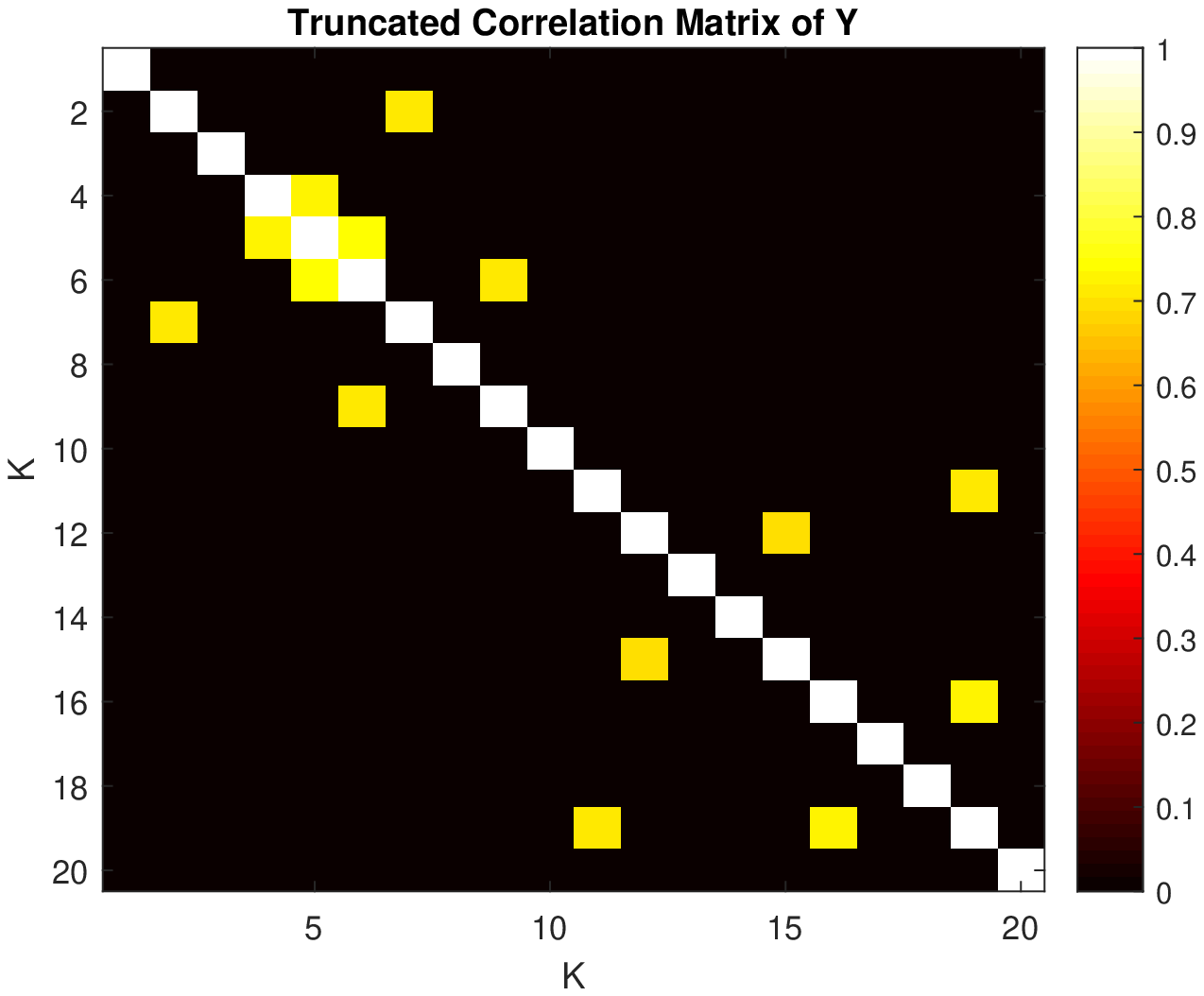}&
\end{tabular}
\caption{(a) is the true regression coefficient matrix in the simulated data for GFLASSO problem, (b) is the estimated regression coefficient matrix by using proximal gradient method (c) is the estimated regression coefficient matrix by using block-coordinate Frank-Wolfe algorithm with away-steps (BCFWAS), (d) is the truncated correlation matrix of the outputs $\mb{Y}$ in the simulated data set base on which the graph in the GFLASSO problem is constructed and (e) is the plot of the logarithmic objective function values of both methods. Median of the 10 sample-paths when running both algorithms are plotted in the lines. The shaded areas shows the upper and lower bounds at each iteration in the 10 replications. The detailed plot in the middle shows that the BCFWAS beats the Prox-Grad after the 20th iteration in this experiment.}
\label{gflasso}
\end{figure*}

Consider the following multivariate linear model:
\begin{align*}
\mathbf{y}_i = \mathbf{X}\boldsymbol\beta_i + \boldsymbol\epsilon_i \quad\quad i = 1,\ldots, N
\end{align*}
where $\mathbf{X} \in \mathbb{R}^{N \times J}$ denotes the matrix of input data for $J$ covariates over $N$ samples, $\mathbf{Y}=[\mathbf{y}_1, \ldots, \mathbf{y}_K] \in \mathbb{R}^{N \times K}$ denotes the matrix of output data for $K$ tasks, $\mathbf{B} = [\boldsymbol\beta_1, \ldots, \boldsymbol\beta_K] \in \mathbb{R}^{J \times K}$ denotes the matrix of regression coefficients for the $K$ tasks and the $\boldsymbol\epsilon_i$'s denote the noise terms that are independent and identically distributed. Let $G = (V, E)$ be a graph where $V$ is the set of vertices and $E$ is the set of edges. In \citep{CLKCX11}, the graph is constructed by taking each task as a vertex and each non-zero off-diagonal entry of the correlation matrix of $\mb{Y}$, as an edge. When the correlation between two tasks is small, it will be truncated to $0$ when calculating the correlation matrix and hence there won't be edges between such tasks. The graph guided fused LASSO (GFLASSO) problem for multi-task sparse regression problem is formulated as 
\begin{align*}
\min_{\mathbf{B}\in \mathbb{R}^{J \times K}} \frac{1}{2}\norm{\mathbf{Y} - \mathbf{XB}}_F^2 + \Omega(\mathbf{B}) + \lambda\norm{\mathbf{B}}_1,
\end{align*}
where $$\Omega(\mathbf{B}) = \gamma \sum_{e = (m, l) \in E} \vert r_{ml}\vert\sum_{j=1}^J \vert \beta_{jm} - \text{sign}(r_{ml})\beta_{jl}\vert, $$ $r_{ml}$ is the $(m, l)$-th entry of \textbf{cor}($\mathbf{Y}$), the truncated correlation matrix of ${\mathbf{Y}}$ at a pre-determined level, $\norm{\cdot}_F$ denotes the Frobenius norm, and  $\gamma$ and $\lambda$ are the regularization parameters.
The GFLASSO was first proposed by \cite{KSX09} for problems in quantitative trait network in genomic association studies. To solve GFLASSO, \cite{KSX09} formulated it as a quadratic programming problem and proposed using an active-set method. Later \cite{CLKCX11} proposed a smoothing proximal gradient method which is more saclable than the active-set method. We propose to use the block-coordinate semi-stochastic Frank-Wolfe algorithm with away-steps to solve this problem. When the underlying graph is a union of several connected components (as it is the case in most genetic studies) instead of being fully connected, the proposed algorithm can be applied by considering each connected component as a block. Such a structure effectively transforms a large original problem into several small subproblems which are easier to solve. When the regression coefficient matrix $\mb{B}$ is sparse, which is also common in most genetic studies, the sparse update in each step of the Frank-Wolfe type algorithms will also have the advantage of being able to extract information from the data sets efficiently. In our numerical example, we follow the data generation procedure in \cite{CLKCX11} which simulates genetic association mapping data. We set $N = 200$, $J = 50$, $K = 20$, entries of $\mb{X}$ and $\boldsymbol\epsilon_k$'s are generated as standard normal random variables, $\boldsymbol\beta_k$'s are generated such that the correlation matrix of outputs of $\mb{y}_k$'s has a block structure and all non-zero elements of the $\boldsymbol\beta_k$'s are equal to 1. The entries in the correlation matrix of $\mb{Y}$ are truncated to $0$ when their absolute value is smaller than $0.7$. In the illustration, fitting GFLASSO using the proposed algorithm reveal the structures of the coefficients very quickly and the proposed algorithm converges much faster than the Prox-Grad method that was proposed in \cite{CLKCX11}.

\noindent\textbf{Remark:} Initial values for Algorithm \ref{block_cond_grad_1} must be vertices of each `block' polytope and even with a relatively bad initialization, the proposed algorithm has a better performance than the Prox-Grad algorithm in this problem. To obtain the same initialization as in the Prox-Grad algorithm, we can write the initial value as a convex combination of the vertices in each block and set $U_{[i]}^{(1)}$ and $\boldsymbol\mu_{[i]}^{(1)}$ accordingly. It is worth noting that due to the numerical error, the points returned by the oracles may not be the vertices of the constraint polytopes. As a result, we may obtain a smaller function value than the optimal value due to the tiny in-feasibilities. Such problems do not occur when the vertices are integral, which is the case of our next example.

\subsection{Structural Support Vector Machines (SVM)} 
The idea of using the Frank-Wolfe algorithm to solve the dual of structural SVM problems was introduced by \cite{LJJSP13}. Briefly speaking, structural SVM solves multi-label classification problems through a linear classifier (with parameter $\mb{w}$) $h_\mb{w}(\mb{x}) = \arg\!\max_{\mb{y} \in \mc{Y}(\mb{x})}\langle \mb{w}, \phi(\mb{x}, \mb{y}) \rangle$, where $\mb{y}$ is a structured output given an input $\mb{x}$, $\mc{Y}(\mb{x})$ denotes the set of all possible labels for an input $\mb{x}$, $\phi: \mc{X} \times \mc{Y} \rightarrow \mr{R}^d$ is a given feature or basis map, and $\mc{X}$ denotes the set of all possible inputs. To learn $\mb{w}$ from training samples $\{(\mb{x}_i, \mb{y}_i), i = 1, \ldots, n\} $, we solve 
\begin{align*}
\min_{\mb{w}, \boldsymbol\xi} \quad &\frac{\lambda}{2}\norm{\mb{w}} + \frac{1}{n}\sum_{i=1}^n \xi_i \\
\text{s.t.}                    \quad          & \langle \mb{w}, \psi_i(\mb{y}) \rangle \geq L_i(y) -\xi_i \\
                                             & \mb{y} \in \mc{Y}_i \\
                                             & i = 1, 2, \ldots, n,
\end{align*}
where $\psi_i(\mb{y}) = \phi(\mb{x}_i, \mb{y}_i) -\phi(\mb{x}_i, \mb{y})$ is the potential function, $L_i(\mb{y})$ denotes loss incurred by predicting $\mb{y}$ instead of the observed label $\mb{y}_i$, $\mc{Y}_i = \mc{Y}(\mb{x}_i)$,  $\xi$'s are the surrogate losses for the data points and $\lambda$ is a regularization parameter. Now consider the Lagrange dual of the above problem. Let $\alpha_i(\mb{y}) \in \mr{R}$ be the dual variable associated with training sample $i$ and possible output $\mb{y} \in \mc{Y}_i$, $\boldsymbol\alpha_i \in \mr{R}^{m_i}$ be the concatenation of all $\alpha_i(\mb{y})$ over all $ \mb{y} \in\mc{Y}_i$,  and $\boldsymbol\alpha \in R^m$ be the column concatenation of all $\alpha_i(\mb{y})$, where $m_i = \vert \mc{Y}_i \vert$ and $m = \sum_{i=1}^n m_i$. Then the dual problem can be written as
\begin{align*}
\min_{\boldsymbol\alpha} \quad &F(\boldsymbol\alpha) := \frac{\lambda}{2}\norm{\mb{A}\boldsymbol\alpha}^2 -\mb{b}^\top \boldsymbol\alpha \\
\text{s. t.}   \quad                & \boldsymbol\alpha \geq \mb{0} \\
									  & \sum_{\mb{y} \in \mc{Y}_i}\alpha_i(\mb{y}) = 1\\
									  & i = 1, \ldots, n,
\end{align*}
where $A \in \mr{R}^{d \times m}$ whose columns are given by $\psi_i(\mb{y}) /(\lambda n), \mb{y} \in \mc{Y}_i, i = 1, 2, \ldots, n$ and $b \in \mr{R}^m$ whose entries are $L_i(\mb{y})/n, \mb{y} \in \mc{Y}_i, i = 1, 2, \ldots, n$. Given a dual solution $\hat{\boldsymbol\alpha}$, the corresponding primal solution can be retrieved from the relation $\hat{\mb{w}} = \mb{A}\hat{\boldsymbol\alpha}$ which is implied by KKT conditions. Observe that $\norm{\mb{A}\boldsymbol\alpha}^2 = \sum_{i=1}^d (\mb{A}_i \boldsymbol\alpha)^2$ and the constraint set for the dual problem is a Cartesian product of polytopes that can be written as $\mc{M} := \Delta_{\vert \mc{Y}_1 \vert} \times \cdots \times \Delta_{\vert \mc{Y}_n \vert}$ where $\Delta_{\vert\mc{Y}_i\vert}$ is the simplex generated by the elements in $\mc{Y}_i$. This suggests the block coordinate stochastic Frank-Wolfe algorithm with away-steps can be applied to this problem if we have linear oracles for each block which solves
\begin{align*}
\min_{\boldsymbol\alpha_i \in \Delta_{\vert \mc{Y}_i \vert}} \langle \nabla_{[i]} F(\cdot), \boldsymbol\alpha_i \rangle
\end{align*}
where $\nabla_{[i]} F(\cdot)$ denotes the partial derivative of $F(\cdot)$ with respect to $i$-th block of variables.
Since the gradient of $F$, $\nabla F = \lambda \mb{A}^\top \mb{A}\boldsymbol\alpha - \mb{b} = \lambda\mb{A}^\top\mb{w} - \mb{b}$ whose $(i, \mb{y})$-th entry is $\{\langle \mb{w}, \psi_i(\mb{y}) \rangle - L_i(\mb{y})\} / n := -H_i(\mb{y}; \mb{w}) / n$, the above linear oracle is equivalent to the following maximization oracle
\begin{align*}
\max_{\mb{y} \in \mc{Y}_i} H_i(\mb{y}; \mb{w}).
\end{align*}
\begin{algorithm}[tb]
\caption{Block Coordinate Semi-Stochastic Frank-Wolfe Algorithm with Away-Steps for Structural SVM}
\label{block_cond_grad_2}
\begin{algorithmic}[1]
\STATE Let $\mb{y}^{(1)} \in \mc{Y}_1 \times \cdots \times \mc{Y}_n$ where $\mu^{(1)}_{\mb{y}^{(1)}_{[i]}} = 1$, $\mu_\mathbf{v}^{(1)} = 0$ for any $\mathbf{v} \in \mathcal{Y}_i / \{\mb{y}^{(1)}_{[i]}\}$ and $U^{(1)}_{[i]} = \{\mathbf{y}^{(1)}_{[i]}\}$.
\STATE Set $\mb{w}^{(1)}_i = \frac{1}{\lambda n}\psi_i(\mb{y}^{(1)}_{[i]})$, $\ell^{(1)}_i = \frac{1}{n}L_i(\mb{y}^{(1)}_{[i]})$, $\mb{w}^{(1)} = \sum_{i=1}^n \mb{w}^{(1)}_i$, and $\ell^{(1)}= \sum_{i=1}^n \ell^{(1)}_i$.
\FOR{$k = 1, 2, \ldots$}
\STATE Uniformly sample $\mc{J} =\{j_1, \ldots, j_{m_k}\}$ from $\{1, \ldots, d\}$ without replacement.
\STATE Randomly pick $i$ from $\{1, \ldots, n\}$
\STATE Compute $\mathbf{p}^{(k)}_{[i]} = \arg\!\max_{\mb{y} \in \mc{Y}_i}H_i(y; I_\mc{J}\mb{w}^{(k)})$
\STATE Compute $\mathbf{u}^{(k)}_{[i]} \in {\arg\!\max}_{\mathbf{v} \in U^{(k)}_{[i]}}\langle I_\mc{J}\mb{w}^{(k)}, \psi_i(\mathbf{v}) \rangle - L_i(\mb{v})$. 
	\IF{$\langle I_\mc{J}\mb{w}^{(k)}, \frac{1}{\lambda n}\psi_i(\mb{p}^\k_{[i]}) - \mb{w}^\k_i \rangle + \ell^{(k)}_i - \frac{1}{n}L_i(\mb{p}_{[i]}^\k) \leq \langle I_\mc{J}\mb{w}^{(k)},  \mb{w}^\k_i - \frac{1}{\lambda n}\psi_i(\mb{u}^\k_{[i]}) \rangle + \frac{1}{n}L_i(\mb{u}_{[i]}^{(k)}) - \ell^{(k)}_i  $}
\STATE Set $\mb{d}^\k_{i} = \frac{1}{\lambda n}\psi_i(\mb{p}^\k_{[i]}) - \mb{w}^\k_i$, ${e}^\k_{i} = \frac{1}{n}L_i(\mb{p}_{[i]}^\k) - \ell_i^\k$ and $\bar{\gamma}^\k_{i} = 1$.
	\ELSE
	\STATE Set $\mathbf{d}^{(k)}_{i} = \mathbf{w}^{(k)}_{i} - \frac{1}{\lambda n}\psi_i(\mathbf{u}^{(k)}_{[i]})$, ${e}^\k_{i} = \ell_i^\k - \frac{1}{n}L_i(\mb{u}_{[i]}^\k) $, and $\bar{\gamma}^{(k)}_{i} = \frac{\mu^{(k)}_{\mathbf{u}^{(k)}_{[i]}}}{1 - \mu^{(k)}_{\mathbf{u}^{(k)}_{[i]}}}$.
	\ENDIF
	\STATE Set $\gamma^\k = \max\{0, \min\{\frac{-\lambda\langle \mb{w}^\k, \mb{d}^\k_{i}\rangle + e^\k_{i}}{\lambda\norm{\mb{d}^k_{i}}^2}, \bar{\gamma}^\k_{i}\} \}$
	\STATE Update $\mb{w}^{(k+1)}_i = \mb{w}^\k + \gamma^\k \mb{d}^\k_{i}$
\ENDFOR
\end{algorithmic}
\end{algorithm}
According to \cite{LJJSP13}, this maximization problem is known as the \textit{loss-augmented decoding} subproblem which can be solved efficiently. The last thing which may prohibit us from applying the proposed algorithm is the potentially exponential number $\vert \mc{Y}_i \vert$ of dual variables due to its combinatorial nature, which makes maintaining a dense $\boldsymbol\alpha$ impossible. This is also the reason why other algorithms, for example QP, become intractable for this problem. However, the simplex structure of the constraint sets enables us to keep the only non-zero coordinate of the solution to the linear oracle which corresponds to the solution of the loss-augmented decoding sub-problem in each iteration. Thus, we only need to keep track of the initial value $\boldsymbol\alpha^{(0)}$ and a list of previously seen solutions to the maximum oracle. Another way to keep track of the vertices is to maintain a list of primal variables from the linear transformation $\mb{w} = \mb{A}\boldsymbol\alpha$ when the dimension of $\mb{w}$ is moderate. Denote the vertex set of $\Delta_{\vert \mc{Y}_i \vert}$ as $\mc{V}_i$ and for an set $\mc{J} \subset \{1, 2, \ldots, d\}$, let $I_\mc{J}$ be a $d \times d$ matrix whose $j$-th diagonal entry equals $1$ for all $j \in \mc{J}$ and all other entries are zero. With above notation, we can apply the stochastic block coordinate Frank-Wolfe algorithm with away-steps the the structural SVM problem.
We apply the above algorithm on the OCR dataset ($n = 6251, d = 4028$) from \cite{TGK03} and compare it with the block-coordinate Frank-Wolfe algorithm. To make a fair comparison, the initial vertex $\mb{y}^{(1)}_i$ in $i$-th coordinate block is chosen to be the observed tag $\mb{y}_i$ corresponding to the choice of primal variables $\mb{w}_i^{(1)} = 0$ as in the implementation of \cite{LJJSP13}. It is worth noting that in the experiments, the performance of both algorithms was worse when the initializations were changed. Computational results for OCR dataset with regularization parameters $\lambda = 0.05, 0.01 $ and $0.001$ are presented here. From the figures, we can see that the stochastic Block-Coordinate Frank-Wolfe Algorithm with Away-Steps (BCFWAS) dominates the Block-Coordinate Frank Wolfe algorithm (BCFW) for every $\lambda$ when an accurate solution is required.
\begin{figure*}
\vskip 0.2in
\begin{center}
\subfigure[]{\includegraphics[scale = 0.55]{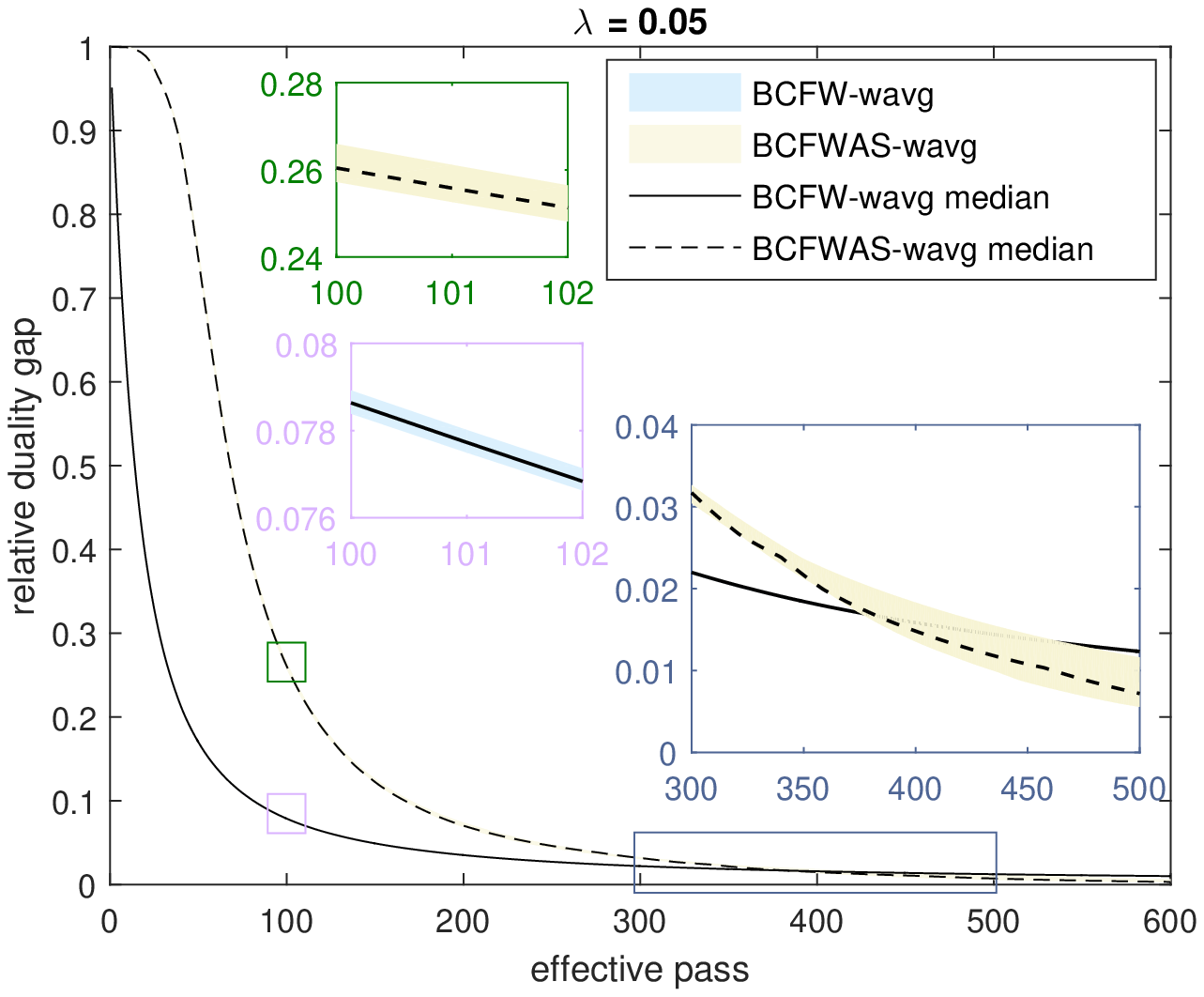}}
\subfigure[]{\includegraphics[scale = 0.55]{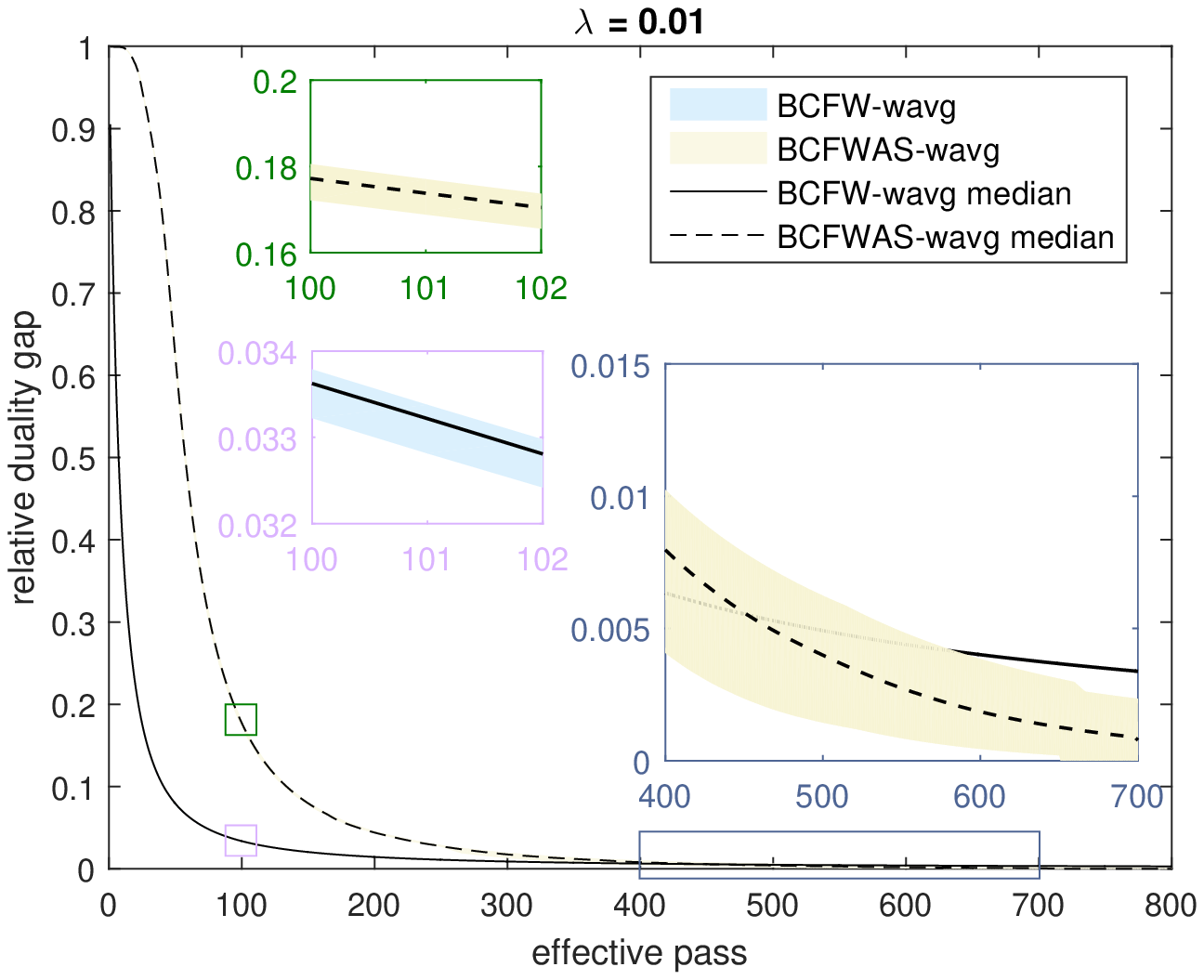}}
\subfigure[]{\includegraphics[scale = 0.55]{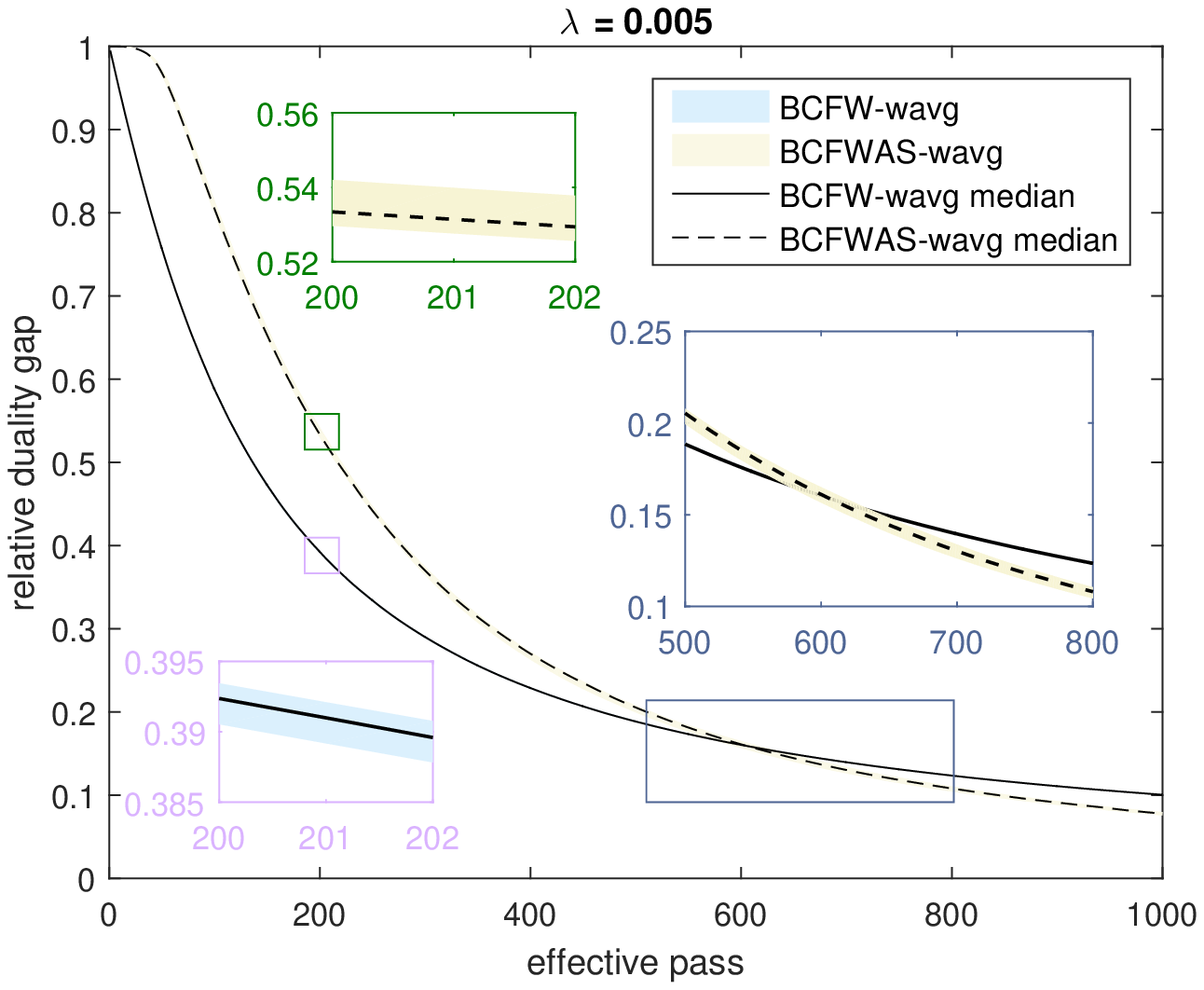}}
\subfigure[]{\includegraphics[scale = 0.55]{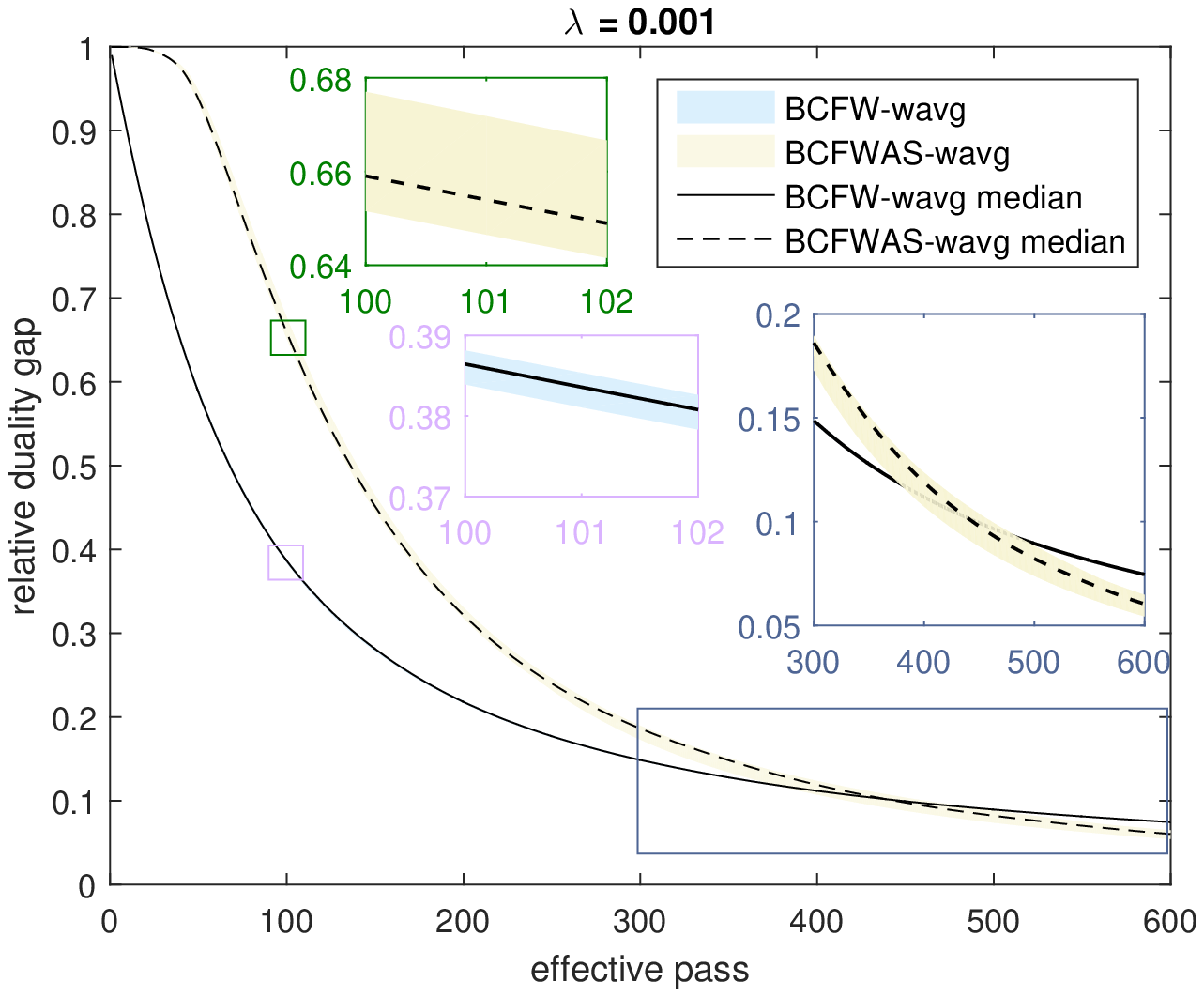}}
\caption{Each figure is plotted the number of effective passes of data versus the relative duality gap. In the implementation, we used the weighted average technique introduced in \cite{LJJSP13} in both algorithms  which outputs the series $\bar{\mb{w}}^{(k+1)} = k/(k+2)\bar{\mb{w}}^{(k)} + 2/(k+2)\mb{w}^{(k+1)}$ and $\bar{\mb{w}^{(0)}} = \mb{w}^{(0)}$. The embedded figures are the corresponding (by frame color) details in the original plots.}
\label{pics:ssvm}
\end{center}
\vskip -0.2in
\end{figure*}
\section{Conclusion, Discussion and Future Work}
In this paper, we established the linear convergence of a semi-stochastic Frank-Wolfe algorithm with away-steps for empirical risk minimization problems and extended it to problems with block-coordinate structure. We applied the algorithms to solve the graph-guided fused LASSO problem and the structural SVM problem. Numerical results indicate the proposed algorithms outperform competing algorithms for these two problems in terms of both iteration cost and number of effective data passes. In addition, the stochastic nature of the proposed algorithms can use an approximate solution to the sub-problems, that is, an inexact oracle which can further reduce the computational cost. Possible extensions of this work include:
\begin{description}
\item[1] In the algorithms, we assume that the Lipschitz constants of the gradient of $F(\cdot)$ are known. This assumption might be avoided by performing back-tracking. 
\item[2] The algorithms are semi-stochastic since we need to use all data to calculate gradients at the final steps. Variance reduced gradient techniques such as the one proposed by \cite{JZ13} might be applicable in stochastic versions of the Frank-Wofe algorithm with away-steps to make it fully stochastic and linearly convergent in high probability.
\item[3] The compact polytope constraint is crucial in the theoretical analysis of the proposed algorithms. Finding ways to relax this assumption is another direction for future work.
\end{description}
\appendix
\section*{Appendix}
\section{Some Useful Lemmas}
The following lemmas are used in the proof of the two theorems in the paper. The first lemma, which is known as the descent lemma, plays a central role in convergence rate analysis. We omit the proof of this lemma since it is a well known result.
\begin{Lemma}[The Descent Lemma] \label{descent_lemma}
Let $f: \mr{R}^n \rightarrow \mr{R}$ be a continuously differentiable function with Lipschitz continuous gradient with constant $L$. Then for any $\mathbf{x}, \mathbf{y} \in \mr{R}^n$ we have 
\begin{align*}
f(\mathbf{y}) \leq f(\mathbf{x}) + \langle \nabla f(\mathbf{x}), \mb{y} - \mb{x} \rangle + \frac{L}{2}\norm{\mb{x} - \mb{y}}^2.
\end{align*}
\end{Lemma}

Hoffman's lemma is useful to bound the distance of a point to the set of optimas. We state the theorem and omit its proof. Please see \cite{Hof52} for references.
\begin{Lemma}[Hoffman's Lemma] \label{hoffman_lemma}
Let $\mathcal{P}$ be a polyhedron defined by $\mathcal{P} = \{\mb{x} \in \mr{R}^p : \mathbf{Cx} \leq \mathbf{d}\}$, for some $\mathbf{C} \in \mr{R}^{m \times p}$ and $\mathbf{d} \in \mathbb{R}^m$, and $\mathcal{S} = \{ \mathbf{x} \in \mathbb{R}^p : {\mathbf{E}}\mathbf{x} = \mathbf{{f}}\}$ where $\mathbf{{E}} \in \mathbb{R}^{r \times p}$ and $\mathbf{{f}} \in \mathbb{R}^r$. Assume that $\mathcal{P} \cap S \neq \emptyset$.
Let $\mathcal{B}$ as the set of all matrices constructed by taking linearly independent rows from the matrix $[\mathbf{{E}}^\top, \mathbf{C}^\top]^\top$, and define $\theta$ as 
\begin{align*}
\theta = \max_{\mathbf{B} \in \mathcal{B}} \frac{1}{\lambda_{\min}(\mathbf{B}\mathbf{B}^\top)}.
\end{align*}
Then for every $\mathbf{x} \in \mathcal{P}$,
\begin{align*}
d(\mathbf{x}, \mathcal{P}\cap \mathcal{S}) \leq \theta \norm{\mathbf{{E}x} - \mathbf{{f}}},
\end{align*}
where $d(\mb{x}, A)$ denotes the distance between a point $x$ and a set $A$, that is, $ d(\mb{x}, A) = \inf\{\norm{\mb{x} - \mb{y}} \; \vert \; \mb{y} \in A\}$.
\end{Lemma}
We will refer to $\theta$ as the Hoffman constant associated with matrix $[\mathbf{{E}}^\top, \mathbf{C}^\top]^\top$.

Recall that in the paper, we focus on developing numerical algorithms for the following empirical risk minimization problem,
\begin{align}
\min_{\mb{x} \in \mc{P}} \Big\{F(\mb{x}) \equiv \frac{1}{n}\sum_{i = 1}^n f_i(\mb{a}^\top_i \mb{x}) + \langle \mb{b}, \mb{x} \rangle \Big\}, \label{general_problem} \tag{P1}
\end{align}
where $\mathcal{P}$ is a non-empty compact polyhedron given by $\mathcal{P} = \{x \in \mathbb{R}^p : \mathbf{Cx} \leq \mb{d}\}$ for some $\mathbf{C} \in \mathbb{R}^{m \times p}$, $\mathbf{d} \in \mathbb{R}^m$. For every $i = 1, \ldots, n$, $\mb{a}_i \in \mr{R}^p$, and $f_i : \mr{R} \rightarrow \mr{R}$ is a strongly convex function with parameter $\sigma_i$ and has a Lipschitz continuous gradient with constant $L_i$. $\mb{b}$ in the linear term of (\ref{general_problem}) is a vector in $\mr{R}^p$. Note that the gradient of $F(\cdot)$ is also Lipschitz continuous with constant $L \leq (\sum_{i=1}^nL_i \norm{a_i}) / n$. However, because of the affine transformation in the argument of each $f_i(\cdot)$, $F(\cdot)$ may not be a strongly convex function.

The next lemma characterizes the set of optimal solution to the general problem (\ref{general_problem}). We follow the arguments in Lemma 14 of \cite{WL14} to prove this result.
\begin{Lemma}\label{unique_lemma}
Write $\mb{A} = [\mb{a}_1, \mb{a}_2, \ldots, \mb{a}_n]^\top$ and let $\mc{P}^{*}$ be the set of optimal solutions to (\ref{general_problem}). Then there exists a constant vector $\mb{t}^*$ and a scalar $s^*$ such that any optimal solution $\mb{x}^* \in \mc{P}^*$ satisfies 
\begin{align*}
\mb{A}\mb{x}^* &= \mb{t}^*\\
\langle \mb{b}, \mb{x}^* \rangle &= s^* \\
\mb{C}\mb{x}^* &\leq \mb{d}.
\end{align*}
Furthermore, $\mb{z}$ satisfies the above three conditions if and only if $\mb{z} \in \mc{P}^*$.
\end{Lemma}
\begin{proof}
For any $\mb{x}_1^*, \mb{x}_2^* \in \mc{P}^*$, $F(\mb{x}_1^*) = F(\mb{x}_2^*)$; hence, by the convexity of $F(\cdot)$,
\begin{align*}
F(\frac{\mb{x}_1^* + \mb{x}^*_2}{2}) = \frac{1}{2}\{F(\mb{x}^*_1) + F(\mb{x}^*_2)\}.
\end{align*}
Therefore
\begin{align*}
\frac{1}{n}\sum_{i=1}^nf_i\{\frac{\mb{a}_i^\top(\mb{x}^*_1 + \mb{x}^*_2)}{2}\} = \frac{1}{n}\sum_{i=1}^n\frac{1}{2}\{f_i(\mb{a}_i^\top \mb{x}_1^*) + f_i(\mb{a_i^\top\mb{x}_2^*})\}.
\end{align*}
By the strong convexity of $f_i(\cdot)$'s, we have $\mb{a}_i^\top\mb{x}^*_1 = \mb{a}^\top_i\mb{x}^*_2$ for every $i = 1, \ldots, n$. Thus $\mb{t}^* = \mb{A}\mb{x}^*$ is unique for any $\mb{x}^* \in \mc{P}^*$. Similarly, since 
\begin{align*}
\frac{1}{n}\sum_{i=1}^n f_i(\mb{a}^\top_i \mb{x}^*_1) + \langle\mb{b}, \mb{x}^*_1 \rangle = F(\mb{x}^*_1) = F(\mb{x}^*_2) = \frac{1}{n}\sum_{i=1}^n f_i(\mb{a}^\top_i \mb{x}^*_2) + \langle\mb{b}, \mb{x}^*_2 \rangle
\end{align*}
and $t^*_i = \mb{a}^\top_i\mb{x}^*_1 = \mb{a}^\top_i\mb{x}^*_2$, we have $\langle \mb{b}, \mb{x}_1^* \rangle = \langle \mb{b}, \mb{x}_2^* \rangle$ and hence $s^* = \langle \mb{b}, \mb{x}^*\rangle$ is unique for any $\mb{x}^* \in \mc{P}^*$. $\mb{C}\mb{x}^* \leq \mb{d}$ follows from $x^* \in \mc{P}^* \subset \mc{P}$.

Conversely, if $\mb{z}$ satisfies $\mb{Az} = \mb{t}^*$, $\langle\mb{b}, \mb{z}\rangle = s^*$ and $\mb{Cz} \leq \mb{d}$, then
$\mb{z} \in \mc{P}$ and $F(\mb{z}) = F(\mb{x}^*)$ for every $\mb{x}^* \in \mc{P}^*$. Therefore $\mb{z} \in \mc{P}^*$.
\end{proof}
The next two lemmas provide a bound on the distance between a feasible solution and the set of optimas. This result is important for the later developments since it can be treated as an alternative to the more stringent assumption that $F(\cdot)$ is strongly convex.
\begin{Lemma} \label{upper_bound_lemma}
Let $F^*$ be the optimal value of problem (\ref{general_problem}) and write $\mb{G}(\mb{x})= (1/n)[f_1'(\mb{a}_1^\top\mb{x}), \ldots, f_n'(\mb{a}_n^\top\mb{x})]^\top $. Then for any $\mb{x} \in \mc{P}$,  $F(\mb{x}) - F^* \leq C_1$, where $C_1 = GD\norm{\mb{A}} + D\norm{\mb{b}}$ and $G = \sup\{\norm{\mb{G}(\mb{x})}\, \vert\, \mb{x} \in \mc{P}\}$.
\end{Lemma}
\begin{proof}
By continuity of $f_i'(\cdot)$ and compactness of $\mc{P}$, we have $G < \infty$. Convexity of $F(\cdot)$ implies
\begin{align*}
F(x) - F^* &\leq \langle \nabla F(\mb{x}), \mb{x} - \mb{x}^*\rangle \\
               &= \langle \frac{1}{n}\sum_{i=1}^n f_i'(\mb{a}_i^\top \mb{x})\mb{a}_i, \mb{x} - \mb{x}^*\rangle + \langle \mb{b}, \mb{x} - \mb{x}^* \rangle \\ 
               &=\langle \mb{A}^\top \mb{G}(\mb{x}), \mb{x} - \mb{x}^*\rangle + \langle \mb{b}, \mb{x} - \mb{x}^*\rangle\\
               &=\langle \mb{G}(\mb{x}), \mb{A}(\mb{x} - \mb{x}^*)\rangle + \langle \mb{b}, \mb{x} - \mb{x}^*\rangle \\
               &\leq \norm{\mb{G}(\mb{x})}\norm{\mb{A}(\mb{x} - \mb{x}^*)} + \norm{\mb{b}}\norm{\mb{x} - \mb{x}^*}\\
               &\leq GD\norm{\mb{A}} + D\norm{\mb{b}} = C_1.
\end{align*}
\end{proof}

\begin{Lemma}\label{strong_cvx_lemma}
For any $\mathbf{x} \in \mathcal{P}$,  
\begin{align*}
d(\mathbf{x}, \mathcal{P}^*)^2 \leq \kappa (F(\mathbf{x}) - F^*),
\end{align*}
where $\kappa =\theta^2 \{C_1 + 2GD\norm{A} + \frac{2n}{\sigma_F}(G^2+1)\}$, $\sigma_F = \min\{\sigma_1, \ldots, \sigma_n\}$, and $\theta$ is the Hoffman constant associated with the matrix $[\mathbf{C}^\top, \mathbf{A}^\top, \mb{b}^\top]^\top$.
\end{Lemma}
\begin{proof}
By Lemma \ref{unique_lemma}, the optimal solution set $\mc{P}^*$ can be defined as $\mc{P}^* = \mc{P} \cap \mc{S}$, where $\mc{S} = \{\mb{x} \in \mr{R}^p \, \vert \, \mb{A}\mb{x} = \mb{t}^*, \langle\mb{b}, \mb{x}\rangle = s^*\}$ for some $\mb{t}^* \in \mr{R}^n$ and $s^* \in \mr{R}$. For any $\mb{x} \in \mc{P}$ , we have
\begin{align*}
d(\mb{x}, \mc{P}^*)^2 \leq \theta^2\{(\langle \mb{b}, \mb{x}\rangle - s^*)^2 + \norm{\mb{A}\mb{x} - \mb{t}^*}^2\} \tag{by lemma \ref{hoffman_lemma}},
\end{align*}
where $\theta$ is the Hoffman's constant associated with matrix $[\mb{c}^\top, \mb{A}^\top, \mb{b}]^\top$. Now, let $\mb{x} \in \mc{P}$ and $\mb{x}^* \in \mc{P}^*$. By strong convexity of the $f_i$'s, we have
\begin{align} \label{strong_expansion}
\frac{1}{n}\sum_{i=1}^n \{f_i(\mb{a}_i^\top \mb{x}) -f_i(\mb{a}_i^\top\mb{x}^*)\} \geq \langle \frac{1}{n}\sum_{i=1}^n f_i'(\mb{a}_i^\top \mb{x}^*)\mb{a}_i, \mb{x} - \mb{x}^*\rangle + \frac{1}{n}\sum_{i=1}^n\frac{\sigma_i}{2}\{\mb{a}_i^\top(\mb{x} - \mb{x}^*)\}^2.
\end{align}
By the first-order optimality conditions, we have $\langle \nabla F(\mb{x}^*), \mb{x} - \mb{x}^* \rangle \geq 0$. Therefore
\begin{align}
 \frac{1}{n}\sum_{i=1}^n\frac{\sigma_i}{2}(\mb{a}^\top _i \mb{x} - t^*_i)^2 &= \frac{1}{n}\sum_{i=1}^n\frac{\sigma_i}{2}\{\mb{a}_i^\top(\mb{x} - \mb{x}^*)\}^2\nonumber \\
 &\leq \langle \nabla F(\mb{x}^*), \mb{x}-\mb{x}^*\rangle + \frac{1}{n}\sum_{i=1}^n\frac{\sigma_i}{2}\{\mb{a}^\top_i (\mb{x} - \mb{x}^*)\}^2 \nonumber \\
					&= \langle \frac{1}{n}\sum_{i=1}^n f'_i(\mb{a}_i^\top \mb{x}^*)\mb{a}_i + \mb{b}, \mb{x} - \mb{x}^* \rangle + \frac{1}{n}\sum_{i=1}^n \frac{\sigma_i}{2}\{\mb{a}_i^\top(\mb{x} - \mb{x}^*)\}^2 \nonumber \\
					&\leq \frac{1}{n}\sum_{i=1}^n\{f_i(\mb{a}_i^\top\mb{x}) - f_i(\mb{a}_i^\top\mb{x}^*)\} + \langle \mb{b}, \mb{x} - \mb{x}^*\rangle \tag{by (\ref{strong_expansion})} \nonumber\\
					& = F(\mb{x}) - F(\mb{x}^*). \label{matrix_bound}
\end{align}
Hence, we can bound $\norm{\mb{A}\mb{x} - \mb{t}^*}^2$ by 
\begin{align*}
\norm{\mb{Ax} - \mb{t}^*}^2  &= \sum_{i=1}^n(\mb{a}_i^\top\mb{x} - t_i^*)^2 \\
&= \frac{n}{2} \frac{1}{n}\sum_{i=1}^n \frac{\sigma_i}{2\sigma_i} (\mb{a}_i^\top \mb{x} - t^*_i)^2 \\
&\leq \frac{n}{2\sigma_F}\frac{1}{n}\sum_{i=1}^n\frac{\sigma_i}{2}(\mb{a}_i^\top \mb{x} - t^*_i)^2\\
&\leq \frac{n}{2\sigma_F}(F(\mb{x}) - F(\mb{x}^*))  \tag{by (\ref{matrix_bound})}.\\
\end{align*}
Next we need to bound $(\langle \mb{b}, \mb{x} \rangle - s^*)^2$ from above. By definition of $s^*$, $F(\cdot)$ and the fact that $\langle \nabla F(\mb{x}^*), \mb{x} - \mb{x}^* \rangle \geq 0$, we have
\begin{align*}
\langle \mb{b}, \mb{x} \rangle - s^* &= \langle \mb{b}, \mb{x} - \mb{x}^* \rangle \\
													 & = \langle \nabla F(\mb{x}^*), \mb{x} - \mb{x}^* \rangle - \langle \frac{1}{n}\sum_{i=1}^n f_i'(\mb{a}_i^\top \mb{x})\mb{a}_i, \mb{x} - \mb{x}^* \rangle \\
													 & \geq -\langle \mb{A}^\top \mb{G}(\mb{x}), \mb{x} - \mb{x}^* \rangle \\
													 & = -\langle \mb{G}(\mb{x}), \mb{A}(\mb{x} - \mb{x}^*) \rangle \\
													 &= - \langle \mb{G}(\mb{x}), \mb{A}\mb{x} - \mb{t}^* \rangle \\
													 &\geq -\norm{\mb{G}(\mb{x})}\norm{\mb{Ax} - \mb{t}^*} \tag{by Cauchy Schwarz inequality}.
\end{align*}
On the other hand, by the convexity of $F(\cdot)$, we have $F(\mb{x}) - F^* \geq \langle \nabla F(\mb{x}^*), \mb{x} - \mb{x}^* \rangle$. Hence
\begin{align*}
\langle \mb{b}, \mb{x} \rangle - s^* &= \langle \nabla F(\mb{x}^*), \mb{x} - \mb{x}^* \rangle - \langle \frac{1}{n}\sum_{i=1}^n f_i'(\mb{a}_i^\top \mb{x})\mb{a}_i, \mb{x} - \mb{x}^* \rangle \\
													 &= \langle \nabla F(\mb{x}^*), \mb{x} - \mb{x}^* \rangle - \langle \mb{A}^\top \mb{G}(\mb{x}^*), \mb{x} - \mb{x}^* \rangle \\
													 &\leq F(\mb{x}) - F^* -\langle \mb{G}(\mb{x}^*), \mb{A}\mb{x} - \mb{t}^* \rangle \\
													 &\leq F(\mb{x}) - F^* + \norm{\mb{G}(\mb{x}^*)}\norm{\mb{A}\mb{x} - \mb{t}^*}.
\end{align*}
Then,
\begin{align*}
(\langle \mb{b},\mb{x}\rangle - s^*)^2 &\leq (F(\mb{x}) - F^* + \norm{\mb{G}(\mb{x}^*)}\norm{\mb{A}\mb{x} - \mb{t}^*})^2 \\
															 &\leq (F(\mb{x}) - F^*)^2 + 2\norm{\mb{G}(\mb{x}^*)}\norm{\mb{A}(\mb{x} - \mb{x}^*)}(F(\mb{x}) - F^*) + \norm{\mb{G}(\mb{x}^*)}^2\norm{\mb{A}(\mb{x} - \mb{x}^*)}^2 \\
															 & \leq (F(\mb{x}) - F^*)C_1 + 2GD\norm{A}(F(\mb{x}) - F^*) + \frac{2nG^2}{\sigma_F}(F(\mb{x}) - F^*)\\
															 &=(C_1 + 2GD\norm{A} + \frac{2nG^2}{\sigma_F})(F(\mb{x}) - F^*).
\end{align*}
Therefore, we have
\begin{align*}
d(\mb{x}, \mc{P}^*) &\leq \theta^2 \{C_1 + 2GD\norm{A} + \frac{2n}{\sigma_F}(G^2+1)\}(F(\mb{x}) - F^*) \\
								& = \kappa (F(\mb{x}) - F^*).
\end{align*}
\end{proof}

In the following two lemmas and corollary, we denote by the index set of active constraints at $\mb{x}$ $I(\mb{x})$, i.e.,  $I(\mb{x}) = \{i \in {1, \ldots, m} \vert \mb{C}_i\mb{x} = d_i\}$. In a similar way, we define the set of active constraints for a set $U$ by  $I(U) = \{i \in \{1, \ldots, m\} \quad \vert \quad \mb{C}_i\mb{v} = d_i, \forall \mb{v} \in U\} = \cap_{\mb{v} \in U}I(\mb{v})$. The standard unit simplex in $\mathbb{R}^n$ is denoted by $\Delta_n =\{\mathbf{x} \in \mathbb{R}^n_+ : \sum_{i = 1}^n x_i = 1\}$ and its relative interior by $\Delta^+_n = \{\mathbf{x} \in \mathbb{R}^n_{++} : \sum_{i = 1}^n x_i = 1\}$.  Given a set $X \subset \mathbf{R}^n$, its convex hull is denoted by $\text{conv}(X)$. Given a convex set $C$, the set of all its extreme points is denoted by $\text{ext}(C)$. Proofs of the following two lemmas can be found in \cite{BS15}.
\begin{Lemma} \label{critical_lemma}
Given $U \subset V$ and $\mb{c} \in \mathbb{R}^p$. If there exists a $\mathbf{z} \in \mathbb{R}^p$ such that $\mathbf{C}_{I(U)}\mathbf{z} \leq 0$ and $\langle \mathbf{c}, \mathbf{z} \rangle > 0$, then 
\begin{align*}
\max_{\mathbf{p} \in V, \mathbf{u} \in U}\langle \mathbf{c}, \mathbf{p - u} \rangle \geq \frac{\Omega_\mathcal{P}}{\vert U \vert} \frac{\langle \mathbf{c}, \mathbf{z} \rangle}{\norm{\mathbf{z}}}
\end{align*}
where 
\begin{align*}
 \Omega_\mathcal{P} = \frac{\zeta}{\phi}
\end{align*}
for
\begin{align*}
\zeta &= \min_{\mathbf{v} \in V, i \in \{1, \ldots, m\}: a_i > \mathbf{C}_i \mathbf{v}} (d_i - \mathbf{C}_i\mathbf{v}), \\
\phi &= \max_{i \in \{1, \ldots, m\} / I(V)}\norm{\mathbf{C}_i}.
\end{align*}
\end{Lemma}

\begin{Lemma}\label{index_lemma}
Let $\mathbf{x} \in \mathcal{K}$ and the set $U \subset V$ satisfy $\mathbf{x} =  \sum_{\mathbf{v} \in U} \mu_\mathbf{v} \mathbf{v}$, where $\mathbf{\mu} \in \Delta^+_{\vert U \vert}$. Then $I(\mathbf{x}) = I(U)$.
\end{Lemma}
\begin{Corollary}\label{linear_approx_lower_bound}
For any $\mathbf{x} \in \mathcal{P} / \mathcal{P}^*$ that can be represented as $\mathbf{x} = \sum_{\mathbf{v} \in U}\mu_{\mathbf{v}}\mathbf{v}$ for some $\mathbf{\mu} \in \Delta_{\vert U \vert}^+$ and $U \subset V$, it holds that, 
\begin{align*}
\max_{\mathbf{u} \in U, \mathbf{p} \in V}\langle \nabla f(\mathbf{x}), \mathbf{u - p} \rangle \geq \frac{\Omega_{\mathcal{K}}}{\vert U \vert} \frac{\langle \nabla f(\mathbf{x}), \mathbf{x} - \mathbf{x}^* \rangle}{\norm{\mathbf{x - x^*}}}.
\end{align*}
\end{Corollary}

\section{Proof of Theorem 1}
\begin{proof}
Let $\delta_{\mc{J}}$ denote the indicator function that 
\begin{align*}
\delta_{\mc{J}}(i) = \left\{ \begin{array}{ll}
         1 & \mbox{if $i \in \mc{J}$};\\
        0 & \mbox{if $i \not\in \mc{J}$}.\end{array} \right.
\end{align*}
First note that the sampling scheme ensures the following two properties of $\mb{g}^\k$:
\begin{align*}
\mr{E}(\mb{g}^\k \: \vert \: \mc{F}^\k) &= \frac{1}{m_k}\sum_{i = 1}^n\mr{E}\{\delta_{(i \in \mc{J})} \: \vert \: \mc{F}^\k\} f'_i(\mb{a}_i^\top \mb{x}^\k)\mb{a}_i \\
							&=\frac{1}{m_k}\sum_{i = 1}^n\frac{m_k}{n} f'_i(\mb{a}_i^\top \mb{x}^\k)\mb{a}_i \\
							&=\nabla F(\mb{x}^\k) \\
\end{align*}
and for $m_k \leq n$
\begin{align*}
\mr{E}\{\norm{\mb{g}^\k - \nabla F(\mb{x}^\k)}^2 \: \vert \: \mc{F}^{(k)}\} &=\frac{1}{m_k^2}[\sum_{i=1}^n \frac{m_k}{n}\{f_i'(\mb{a}_i^\top\mb{x}^\k)\}^2\mb{a}_i^\top\mb{a}_i+ \sum_{i \neq j}\frac{m_k(m_k -1)}{n(n-1)}f_i'(\mb{a}_i^\top\mb{x}^\k)f_j'(\mb{a}_i^\top\mb{x}^\k)\mb{a}^\top_i\mb{a}_j]\\
&\quad - \{\frac{1}{n}\sum_{i = 1}^n f'_i(\mb{a}_i^\top\mb{x}^\k)\mb{a}_i\}^\top\{\frac{1}{n}\sum_{i = 1}^n f'_i(\mb{a}_i^\top\mb{x}^\k)\mb{a}_i\}\\
&= (\frac{1}{m_k} - \frac{1}{n})[\frac{1}{n}\sum_{i=1}^n \{f_i'(\mb{a}_i^\top\mb{x}^\k)\}^2\mb{a}_i^\top\mb{a}_i - \frac{1}{n(n-1)}\sum_{i \neq j}f_i'(\mb{a}_i^\top\mb{x}^\k)f_j'(\mb{a}_i^\top\mb{x}^\k)\mb{a}^\top_i\mb{a}_j]\\
&\leq (\frac{1}{m_k} - \frac{1}{n})C_2.
\end{align*}
Hence we also have the unconditional version of the above relations, i.e.,
\begin{align*}
\mr{E}(\mb{g}^\k) = \mr{E}\{\nabla F(\mb{x}^\k)\}, 
\end{align*}
and 
\begin{align*}
\mr{E}\{\norm{\mb{g}^\k - \nabla F(\mb{x}^\k)}^2\} \leq (\frac{1}{m_k} - \frac{1}{n})C_2. 
\end{align*}

From the choice of $\mb{d}^{(k)}$ in the algorithm, we obtain
\begin{align*}
\langle \mb{g}^{(k)}, \mb{d}^{(k)} \rangle \leq \frac{1}{2}(\langle \mb{g}^{(k)}, \mb{p}^\k - \mb{x}^\k \rangle + \langle \mb{g}^{(k)}, \mb{x}^\k - \mb{u}^\k \rangle) = \frac{1}{2}\langle \mb{g}^{(k)}, \mb{p}^{(k)} - \mb{u}^{(k)} \rangle \leq 0.
\end{align*}
Hence, we can lower bound $\mr{E}\langle \mb{g}^{(k)}, \mb{d}^{(k)} \rangle^2$ by 
\begin{align*}
\mr{E}\langle \mb{g}^{(k)}, \mb{d}^{(k)} \rangle^2 &\geq \frac{1}{4} \mr{E}\langle \mb{g}^{(k)}, \mb{u}^{(k)} - \mb{p}^\k \rangle^2 \\
&=\frac{1}{4} \mr{E}\{\mr{E}(\langle \mb{g}^{(k)}, \mb{u}^{(k)} - \mb{p}^\k \rangle^2 \vert \mc{F}^\k)\} \tag{\text{definition of conditional expectation}} \\
&\geq \frac{1}{4} \mr{E}\{\mr{E}(\langle \mb{g}^{(k)}, \mb{u}^{(k)} - \mb{p}^\k \rangle \vert \mc{F}^\k)\}^2 \tag{\text{Jensen's inequality}} \\
&=  \frac{1}{4}\mr{E}\{\mr{E}(\max_{\mb{p} \in V, \mb{u} \in U^\k}\langle \mb{g}^{(k)}, \mb{u} - \mb{p} \rangle \vert \mc{F}^\k)\}^2 \tag{definition of $\mb{p}^\k$ and $\mb{u}^\k$} \\
&\geq \frac{1}{4} \mr{E}\{\max_{\mb{p} \in V, \mb{u} \in U^\k}\mr{E}(\langle \mb{g}^{(k)}, \mb{u} - \mb{p} \rangle \vert \mc{F}^\k)\}^2 \tag{Jensen's inequality} \\
& = \frac{1}{4} \mr{E}\{\max_{\mb{p} \in V, \mb{u} \in U^\k}\langle \nabla F(\mb{x}^\k), \mb{u} - \mb{p} \rangle\}^2 \tag{unbiasedness of stochastic gradient}\\
& \geq \frac{1}{4} \mr{E}\{\frac{\Omega_{\mathcal{P}}}{\vert U^{(k)} \vert} \frac{\langle \nabla F(\mathbf{x}^\k), \mathbf{x}^\k - \mathbf{x}^* \rangle}{\norm{\mathbf{x}^\k - \mb{x}^*}}\}^2 \tag{by Corollary \ref{linear_approx_lower_bound}} \\
& \geq \frac{\Omega_{\mathcal{P}}^2}{4N^2 \kappa} \mr{E}\{F(\mb{x}^\k) - F^*\} \tag{by Lemma \ref{strong_cvx_lemma} and $\vert U^{(k)} \vert \leq N$}.
\end{align*}
Similarly, we can upper bound $\mr{E}\langle \mb{g}^{(k)}, \mb{d}^{(k)} \rangle$ by 
\begin{align*}
\mr{E}\langle \mb{g}^{(k)}, \mb{d}^{(k)} \rangle & \leq \frac{1}{2} \mr{E}\langle \mb{g}^{(k)}, \mb{p}^{(k)} - \mb{u}^{(k)} \rangle \\
& \leq  \frac{1}{2} \mr{E}\langle \mb{g}^{(k)}, \mb{x}^* - \mb{x}^{(k)}\rangle \tag{definition of $\mb{p}^\k$ and $\mb{u}^\k$} \\
& = \frac{1}{2}\mr{E}\langle \nabla F(\mb{x}^\k), \mb{x}^* - \mb{x}^\k \rangle \tag{unbiasedness of stochastic gradient}\\
& \leq \frac{1}{2}\mr{E}\{F(\mb{x}^*) - F(\mb{x}^\k)\} \tag{Convexity of $F(\cdot)$}.
\end{align*}
With the above bounds, we can separate our analysis into the following four cases at iteration $k$
\begin{description}
\item[($A^{(k)}$)]\hspace{0.5cm} $\gamma_{\max}^\k \geq 1$ and $\gamma^\k \leq 1$ .
\item[($B^{(k)}$)]\hspace{0.5cm} $\gamma_{\max}^\k \geq 1$ and $\gamma^\k \geq 1$.
\item[($C^{(k)}$)]\hspace{0.5cm} $\gamma_{\max}^\k < 1$ and $\gamma^\k < \gamma_{\max}^\k$.
\item[($D^{(k)}$)]\hspace{0.5cm} $\gamma_{\max}^\k < 1$ and $\gamma^\k = \gamma_{\max}^\k$.
\end{description}
By the descent lemma, we have 
\begin{align}
F(\mathbf{x}^{(k + 1)}) = F(\mathbf{x}^{(k)} + \gamma^{(k)}\mathbf{d}^{(k)}) &\leq F(\mathbf{x}^{(k)}) + \gamma^{(k)}\langle \nabla F(\mathbf{x^{(k)}}), \mathbf{d}^{(k)} \rangle + \frac{L(\gamma^{(k)})^2}{2}\norm{\mathbf{d^{(k)}}}^2 \nonumber \\
&= F(\mathbf{x}^{(k)}) + \gamma^{(k)}\langle \mb{g}^{(k)}, \mathbf{d}^{(k)} \rangle + \frac{L(\gamma^{(k)})^2}{2}\norm{\mathbf{d^{(k)}}}^2 + \gamma^{(k)}\langle \nabla F(\mathbf{x^{(k)}}) - \mb{g}^{(k)}, \mathbf{d}^{(k)} \rangle. \label{basic_neq}
\end{align}

\noindent In case $(A^{(k)})$, let $\delta_{A^{(k)}}$ denote the indicator function for this case. Then
\begin{align*}
\delta_{A^{(k)}}\{F(\mathbf{x}^{(k + 1)}) - F^*\} &\leq \delta_{A^{(k)}} \{ F(\mathbf{x}^{(k)}) - F^* + \gamma^{(k)}\langle \mb{g}^{(k)}, \mathbf{d}^{(k)} \rangle + \frac{L(\gamma^{(k)})^2}{2}\norm{\mathbf{d^{(k)}}}^2 + \gamma^{(k)}\langle \nabla F(\mathbf{x^{(k)}}) - \mb{g}^{(k)}, \mathbf{d}^{(k)} \rangle \} \\
&= \delta_{A^{(k)}}\{F(\mathbf{x}^{(k)}) - F^* - \frac{\langle \mb{g}^\k, \mb{d}^\k \rangle^2}{2L\norm{\mb{d}^\k}^2} + \frac{\langle \mb{g}^\k, -\mb{d}^\k \rangle}{L\norm{\mb{d}^\k}^2}\langle \nabla F(\mb{x}^\k) -\mb{g}^\k, \mb{d}^\k \rangle \}\tag{definition of $\gamma^\k$ in case $A^{(k)}$}  \\
&\leq \delta_{A^{(k)}}\{ F(\mathbf{x}^{(k)}) - F^* -  \frac{\langle \mb{g}^\k, \mb{d}^\k \rangle^2}{2L\norm{\mb{d}^\k}^2} + \frac{\norm{\mb{g}^\k}\norm{\mb{d}^\k}}{L\norm{\mb{d}^\k}^2}\norm{\nabla F(\mb{x}^\k) -\mb{g}^\k}\norm{\mb{d}^\k }\} \tag{Cauchy-Schwarz}\\
&\leq \delta_{A^{(k)}} \{ F(\mathbf{x}^{(k)}) - F^* -  \frac{\langle \mb{g}^\k, \mb{d}^\k \rangle^2}{2LD^2} + \frac{G}{L}\norm{\nabla F(\mb{x}^\k) -\mb{g}^\k}\} \\
&\leq \delta_{A^{(k)}} \{ F(\mathbf{x}^{(k)}) - F^* -  \frac{\langle \mb{g}^\k, \mb{d}^\k \rangle^2}{2D\max (G, LD)} + \frac{G}{L}\norm{\nabla F(\mb{x}^\k) -\mb{g}^\k}\}
\end{align*}

\noindent In case $(B^{(k)})$, since $\gamma^\k > 1$, we have 
\begin{align}
&-\langle \mb{g}^\k, \mb{d}^\k \rangle > L \norm{\mb{d}^\k}^2 \quad \quad \quad \text{and} \label{case_b_in_prod_bound}\\
&\gamma^{(k)}\langle \mb{g}^{(k)}, \mathbf{d}^{(k)} \rangle + \frac{L(\gamma^{(k)})^2}{2}\norm{\mathbf{d^{(k)}}}^2 \leq \langle \mb{g}^{(k)}, \mathbf{d}^{(k)} \rangle + \frac{L}{2}\norm{\mathbf{d^{(k)}}}^2. \label{case_b_monotone}
\end{align}
Use $\delta_{B^{(k)}}$ to denote the indicator function for this case. Then,
\begin{align*}
\delta_{B^{(k)}} \{F(\mathbf{x}^{(k + 1)}) - F^* \}&\leq \delta_{B^{(k)}} \{F(\mathbf{x}^{(k)}) - F^* + \gamma^{(k)}\langle \nabla F(\mathbf{x^{(k)}}), \mathbf{d}^{(k)} \rangle + \frac{L(\gamma^{(k)})^2}{2}\norm{\mathbf{d^{(k)}}}^2\}  \\
&= \delta_{B^{(k)}} \{F(\mathbf{x}^{(k)}) - F^* + \gamma^{(k)}\langle \mb{g}^{(k)}, \mathbf{d}^{(k)} \rangle + \frac{L(\gamma^{(k)})^2}{2}\norm{\mathbf{d^{(k)}}}^2 + \gamma^{(k)}\langle \nabla F(\mathbf{x^{(k)}}) - \mb{g}^{(k)}, \mathbf{d}^{(k)} \rangle\}\\
&\leq \delta_{B^{(k)}} \{F(\mathbf{x}^{(k)}) - F^* +\langle \mb{g}^{(k)}, \mathbf{d}^{(k)} \rangle + \frac{L}{2}\norm{\mathbf{d^{(k)}}}^2 + \gamma^{(k)}\langle \nabla F(\mathbf{x^{(k)}}) - \mb{g}^{(k)}, \mathbf{d}^{(k)} \rangle \} \tag{ by (\ref{case_b_monotone})}\\
&\leq \delta_{B^{(k)}} \{ F(\mathbf{x}^{(k)}) - F^* + \frac{1}{2}\langle \mb{g}^{(k)}, \mathbf{d}^{(k)} \rangle +\frac{\langle \mb{g}^{(k)}, -\mb{d}^{(k)}\rangle}{L\norm{\mb{d}^{(k)}}^2} \langle \nabla F(\mathbf{x^{(k)}}) - \mb{g}^{(k)}, \mathbf{d}^{(k)} \rangle \} \tag{by (\ref{case_b_in_prod_bound})} \\
&\leq \delta_{B^{(k)}} \{ F(\mathbf{x}^{(k)}) - F^* + \frac{1}{2}\langle \mb{g}^{(k)}, \mathbf{d}^{(k)} \rangle + \frac{G}{L}\norm{\nabla F(\mathbf{x^{(k)}}) - \mb{g}^{(k)}} \} \tag{Cauchy-Schwarz} \\
&=\delta_{B^{(k)}}\{ F(\mathbf{x}^{(k)}) - F^* - \frac{GD}{2} \frac{\langle -\mb{g}^{(k)}, \mathbf{d}^{(k)} \rangle}{GD} + \frac{G}{L}\norm{\nabla F(\mathbf{x^{(k)}}) - \mb{g}^{(k)}} \}
\end{align*}
Note that $0 \leq \langle -\mb{g}^{(k)}, \mathbf{d}^{(k)} \rangle /(GD) \leq 1$, as a result,
\begin{align*}
\delta_{B^{(k)}} \{F(\mathbf{x}^{(k + 1)}) - F^* \} &\leq \delta_{B^{(k)}}\{ F(\mathbf{x}^{(k)}) - F^* - \frac{GD}{2} \frac{\langle -\mb{g}^{(k)}, \mathbf{d}^{(k)} \rangle^2}{G^2D^2} + \frac{G}{L}\norm{\nabla F(\mathbf{x^{(k)}}) - \mb{g}^{(k)}} \} \\
&= \delta_{B^{(k)}}\{ F(\mathbf{x}^{(k)}) - F^* - \frac{\langle \mb{g}^{(k)}, \mathbf{d}^{(k)} \rangle^2}{2GD} + \frac{G}{L}\norm{\nabla F(\mathbf{x^{(k)}}) - \mb{g}^{(k)}} \} \\
&\leq  \delta_{B^{(k)}}\{ F(\mathbf{x}^{(k)}) - F^* - \frac{\langle \mb{g}^{(k)}, \mathbf{d}^{(k)} \rangle^2}{2D\max (G, LD)} + \frac{G}{L}\norm{\nabla F(\mathbf{x^{(k)}}) - \mb{g}^{(k)}} \} \\
\end{align*}
In case $(C^{(k)})$, let $\delta_{C^{(k)}}$ be the indicator function for this case and we can use exactly the same argument as in case (A) to obtain the following inequality 
\begin{align*}
\delta_{C^{(k)}} \{F(\mathbf{x}^{(k + 1)}) - F(\mb{x}^*)\}  \leq \delta_{C^{(k)}} \{ F(\mathbf{x}^{(k)}) - F^* -  \frac{\langle \mb{g}^\k, \mb{d}^\k \rangle^2}{2D\max (G, LD)} + \frac{G}{L}\norm{\nabla F(\mb{x}^\k) -\mb{g}^\k}\}
\end{align*}

\noindent Case $(D^{(k)})$ is the so called ``drop step" in the conditional gradient algorithm with away-steps. Use $\delta_{D^{(k)}}$ to denote the indicator function for this case. Note that $\gamma^\k = \gamma_{\max}^\k \leq -\langle \mb{g}^\k, \mb{d}^\k \rangle / (L\norm{\mb{d}^\k}^2)$ in this case, then we have
\begin{align*}
\delta_{D^{(k)}} \{(F(\mathbf{x}^{(k + 1)}) - F^*) \} &\leq \delta_{D^{(k)}}\{ F(\mathbf{x}^{(k)}) - F^* + \gamma^{(k)}\langle \nabla F(\mathbf{x^{(k)}}), \mathbf{d}^{(k)} \rangle + \frac{L(\gamma^{(k)})^2}{2}\norm{\mathbf{d^{(k)}}}^2 \} \\
&= \delta_{D^{(k)}}\{ F(\mathbf{x}^{(k)}) - F^* + \gamma^{(k)}\langle \mb{g}^{(k)}, \mathbf{d}^{(k)} \rangle + \frac{L(\gamma^{(k)})^2}{2}\norm{\mathbf{d^{(k)}}}^2 + \gamma^{(k)}\langle \nabla F(\mathbf{x^{(k)}}) - \mb{g}^{(k)}, \mathbf{d}^{(k)} \rangle \}\\
&\leq \delta_{D^{(k)}} \{ F(\mathbf{x}^{(k)}) - F^* + \frac{\gamma^\k}{2}\langle \mb{g}^{(k)}, \mathbf{d}^{(k)} \rangle + \frac{\langle \mb{g}^{(k)}, \mathbf{d}^{(k)} \rangle}{L\norm{\mb{d}^\k}^2}\norm{\nabla F(\mathbf{x^{(k)}}) - \mb{g}^{(k)}}\norm{\mb{d}^\k} \} \\
&\leq \delta_{D^{(k)}} \{ F(\mathbf{x}^{(k)}) - F^* + \frac{G}{L}\norm{\nabla F(\mb{x}^\k) -\mb{g}^\k}\} \\
&= \delta_{D^{(k)}} \{ F(\mathbf{x}^{(k)}) - F^* - \frac{\langle \mb{g}^\k, \mb{d}^\k \rangle^2}{2D\max (G, LD)} + \frac{\langle \mb{g}^\k, \mb{d}^\k \rangle^2}{2D\max (G, LD)} + \frac{G}{L}\norm{\nabla F(\mb{x}^\k) -\mb{g}^\k}\} \\
&\leq \delta_{D^{(k)}} \{ F(\mathbf{x}^{(k)}) - F^* - \frac{\langle \mb{g}^\k, \mb{d}^\k \rangle^2}{2D\max (G, LD)} + \frac{G}{L}\norm{\nabla F(\mb{x}^\k) -\mb{g}^\k} \} + \delta_{D^{(k)}} \frac{G^2}{2L}.
\end{align*}
By the fact that $\delta_{A^{(k)}} + \delta_{B^{(k)}} + \delta_{C^{(k)}} + \delta_{D^{(k)}} = 1$, we have
\begin{align*}
F(\mb{x}^{(k+1)}) - F^* \leq F(\mb{x}^\k) - F^* - \frac{\langle \mb{g}^\k, \mb{d}^\k \rangle^2}{2D\max (G, LD)} + \frac{G}{L}\norm{\nabla F(\mb{x}^\k) -\mb{g}^\k} +  \delta_{D^{(k)}} \frac{G^2}{2L}
\end{align*}
Taking expectation on both sides yields
\begin{align*}
\mr{E}\{F(\mb{x}^{(k+1)}) - F^*\} &\leq \mr{E}\{F(\mb{x}^\k) - F^*\} - \frac{\mr{E}(\langle \mb{g}^\k, \mb{d}^\k \rangle^2)}{2D\max (G, LD)}+ \frac{G}{L} \mr{E}\{\norm{\nabla F(\mb{x}^\k) -\mb{g}^\k}\} + \frac{G^2}{2L}\mr{P}(D^{(k)}) \\
&\leq  \mr{E}\{F(\mb{x}^\k) - F^*\} - \frac{\Omega_{\mathcal{P}}^2}{8N^2 \kappa D\max (G, LD)} \mr{E}\{F(\mb{x}^\k) - F^*\} \\
&\quad  + \frac{G\sqrt{C_2}}{L} \sqrt{(\frac{1}{m_k} - \frac{1}{n})} + \frac{G^2}{2L}\mr{P}(D^{(k)})
\end{align*}

Denote $\rho = \Omega_{\mathcal{P}}^2 /\{8N^2 \kappa D\max (G, LD)\}$. By the definition of $m_k$ and the assumption on $\mr{P}(D^{(k)})$ we have
\begin{align*}
\mathbb{E}\{F(\mathbf{x}^{(k+1)}) - F^*\} &\leq (1 - \rho)\mr{E}\{F(\mathbf{x^{(k)}}) - F^*\} + \frac{G\sqrt{C_2}}{L}(1 - \rho)^{\alpha k} + \frac{G^2}{2L}(1 - \rho)^{\beta k} \\
&\leq (1 - \rho)^k\{F(\mb{x}^{(1)}) - F^*\} + (\frac{G\sqrt{C_2}}{L}(1 - \rho)^{\alpha k}\sum_{m=0}^{k-1}(1 - \rho)^{m(1 - \alpha)} + \frac{G^2}{2L}(1 - \rho)^{\beta k}\sum_{m = 0}^{k-1}(1 - \rho)^{m(1 - \beta)})  \\
&\leq (1 - \rho)^k\{F(\mb{x}^{(1)}) - F^*\} + (\frac{G\sqrt{C_2}}{L}\frac{(1 - \rho)^{\alpha k}}{1 - (1 - \rho)^{1 -\alpha}} + \frac{G^2}{2L}\frac{(1 - \rho)^{\beta k}}{1 - (1 - \rho)^{1 - \beta}} ) \\
&\leq (1 - \rho)^{\min (\alpha, \beta) k}\{F(\mb{x}^{(1)}) - F^*  + \frac{G\sqrt{C_2}}{L\{1 - (1 - \rho)^{1 -\alpha}\}} + \frac{G^2}{2L\{1 - (1 - \rho)^{1 - \beta}\}}\} \\
&= (1 - \rho)^{\min (\alpha, \beta) k} C_3
\end{align*}
\end{proof}
\section{Proof of Theorem 2}
\begin{proof}
Similar to the proof of Theorem 1, the expectation and variance of the stochastic gradient $\mb{g}^{(k)}$ condition on $\mr{F}^\k$ are given by
\begin{align*}
\mr{E}(\mb{g}^\k \vert \mc{F}^\k) = \nabla F(\mb{x}^\k) \\
\end{align*}
and for $m_k \leq n$
\begin{align*}
\mr{E}\{\norm{\mb{g}^\k - \nabla F(\mb{x}^\k)}^2 \vert \mc{F}^\k\}&= (\frac{1}{m_k} - \frac{1}{n})[\frac{1}{n}\sum_{i=1}^n \{f_i'(\mb{a}_i^\top\mb{x}^\k)\}^2\mb{a}_i^\top\mb{a}_i - \frac{1}{n(n-1)}\sum_{i \neq j}f_i'(\mb{a}_i^\top\mb{x}^\k)f_j'(\mb{a}_i^\top\mb{x}^\k)\mb{a}^\top_i\mb{a}_j]\\
&\leq (\frac{1}{m_k} - \frac{1}{n})C_2.
\end{align*}
By the descent lemma, the block coordinate structure and the sampling scheme we have 
\begin{align*}
F(\mathbf{x}^{(k + 1)}) &= F(\mathbf{x}^{(k)} + \sum_{i= 1}^r \gamma^{(k)}_{[l_i]}\mb{U}_{[l_i]}\mathbf{d}^{(k)}_{[l_i]}) \leq F(\mathbf{x}^{(k)}) + \sum_{i= 1}^r \gamma^{(k)}_{[l_i]}\langle \nabla F(\mathbf{x^{(k)}}), \mb{U}_{[l_i]}\mathbf{d}^{(k)}_{[l_i]} \rangle + \frac{L}{2}\norm{\sum_{i= 1}^r \gamma^{(k)}_{[l_i]}\mb{U}_{[l_i]}\mathbf{d}^{(k)}_{[l_i]}}^2 \\
&= F(\mathbf{x}^{(k)}) +\sum_{i= 1}^r \gamma^{(k)}_{[l_i]}\langle \mb{g}^{(k)}, \mb{U}_{[l_i]}\mathbf{d}^{(k)}_{[l_i]} \rangle + \frac{L}{2}\sum_{i=1}^r (\gamma^{(k)}_{[l_i]})^2\norm{\mathbf{d}^{(k)}_{[l_i]}}^2 + \sum_{i= 1}^r \gamma^{(k)}_{[l_i]} \langle \nabla F(\mb{x}^\k) - \mb{g}^\k,  \mb{U}_{[l_i]}\mathbf{d}^{(k)}_{[l_i]} \rangle \\
&= F(\mathbf{x}^{(k)}) +\sum_{i= 1}^r \gamma^{(k)}_{[l_i]}\langle \mb{U}_{[l_i]}^\top\mb{g}^{(k)}, \mathbf{d}^{(k)}_{[l_i]} \rangle + \frac{L}{2}\sum_{i=1}^r (\gamma^{(k)}_{[l_i]})^2\norm{\mathbf{d}^{(k)}_{[l_i]}}^2 + \sum_{i= 1}^r \gamma^{(k)}_{[l_i]} \langle \nabla F(\mb{x}^\k) - \mb{g}^\k,  \mb{U}_{[l_i]}\mathbf{d}^{(k)}_{[l_i]} \rangle.
\end{align*}

Denote $\mb{g}_{[l_i]}^\k = \mb{U}_{[l_i]}^\top \mb{g}^\k$. From the choice of $\mb{d}^\k_{[l_i]}$ in the algorithm,  we have the following inequality
\begin{align*}
\langle \mb{g}^\k_{[l_i]}, \mb{d}^\k_{[l_i]}\rangle &\leq \frac{1}{2}(\langle \mb{g}^\k_{[l_i]}, \mb{p}^{(k)}_{[l_i]} - \mb{x}^\k_{[l_i]} \rangle + \langle \mb{g}^\k_{[l_i]}, \mb{x}^\k_{[l_i]} - \mb{u}^{(k)}_{[l_i]}\rangle) = \frac{1}{2}(\langle \mb{g}^\k_{[l_i]}, \mb{p}^{(k)}_{[l_i]} - \mb{u}^\k_{[l_i]} \rangle) \leq 0.
\end{align*}
Thus, similar to the argument in the proof of Theorem 1, we can derive the following block coordinate bounds
\begin{align*}
\mr{E}(\langle  \mb{g}^\k_{[l_i]}, \mb{d}^\k_{[l_i]}\rangle) &\leq \frac{1}{2}\mr{E}(\langle \mb{g}^\k_{[l_i]}, \mb{p}^{(k)}_{[l_i]}  \rangle)\\
&\leq \frac{1}{2}\mr{E}(\langle \mb{U}_{[l_i]}^\top\mb{g}^\k, \mb{x}^*_{[l_i]} - \mb{x}^{(k)}_{[l_i]} \rangle) \tag{definition of $\mb{p}^\k_{[l_i]}$}\\
&= \frac{1}{2}\mr{E}\langle \mb{U}_{[l_i]}^\top\mb{g}^\k, \mb{U}_{[l_i]}^\top(\mb{x}^* - \mb{x}^{(k)}) \rangle) \\
&= \frac{1}{2}\mr{E}\langle \mb{U}_{[l_i]}\mb{U}_{[l_i]}^\top\mb{g}^\k, \mb{x}^* - \mb{x}^{(k)} \rangle) \\
&= \frac{1}{2q}\mr{E}\langle\mb{g}^\k, \mb{x}^* - \mb{x}^{(k)} \rangle) \tag{independence of sampled gradient and sampled block} \\
&= \frac{1}{2q}\mr{E}\langle\nabla F(\mb{x}^\k), (\mb{x}^* - \mb{x}^{(k)})\rangle \tag{unbiasedness of the stochastic gradient}.
\end{align*}
At iteration $k$, write $U^{(k)} = U_{[l_1]}^\k \times U_{[l_2]}^\k \times \cdots \times U_{[l_q]}^\k$, that is, the set of vertices used to represent $\mb{x}^{(k)}$, then
\begin{align*}
\mr{E}(\langle  \mb{g}^\k_{[l_i]}, \mb{d}^\k_{[l_i]}\rangle^2) &\geq \frac{1}{4}\mr{E}\{ \mr{E}(\langle  \mb{g}^\k_{[l_i]}, \mb{u}^\k_{[l_i]} - \mb{p}^\k_{[l_i]}\rangle^2 \vert \mc{F}^\k)\} \\
&\geq \frac{1}{4}\mr{E}\{\mr{E}(\langle  \mb{g}^\k_{[l_i]}, \mb{u}^\k_{[l_i]} - \mb{p}^\k_{[l_i]}\rangle \vert \mc{F}^\k)\}^2 \tag{Jensen's inequality} \\
&= \frac{1}{4}\mr{E}\{\frac{1}{q}\sum_{i = 1}^q\mr{E}(\langle \mb{g}_{[i]}^\k, \mb{u}^\k_{[i]} - \mb{p}^\k_{[i]}\rangle \vert \mc{F}^\k)\}^2\\
&= \frac{1}{4}\mr{E}\{\frac{1}{q}\sum_{i=1}^q \mr{E}( \max_{\mb{p}_{[i]} \in V_{[i]}, \mb{u}_{[i]} \in U_{[i]}^\k}\langle \mb{g}^\k_{[i]}, \mb{u}_{[i]} - \mb{p}_{[i]}\rangle \vert \mc{F}^\k)\}^2\\
&\geq \frac{1}{4}\mr{E}[\frac{1}{q}\sum_{i=1}^q\max_{\mb{p}_{[i]} \in V_{[i]}, \mb{u}_{[i]} \in U_{[i]}^\k}\{\mr{E}(\langle \mb{g}^\k_{[i]}, \mb{u}_{[i]} - \mb{p}_{[i]}\rangle \vert \mc{F}^\k)\}]^2 \tag{Jensen's inequality} \\
&\geq \frac{1}{4}\mr{E}[\frac{1}{q}\max_{\tilde{\mb{p}} \in V, \tilde{\mb{u}} \in U^\k}\{\sum_{i=1}^q \mr{E}(\langle \mb{U}_{[i]}^\top\mb{g}^\k, \mb{U}_{[i]}^\top(\tilde{\mb{u}} - \tilde{\mb{p}}) \rangle \vert \mc{F}^\k) \}]^2 \\
&= \frac{1}{4}\mr{E}[\frac{1}{q}\max_{\tilde{\mb{p}} \in V, \tilde{\mb{u}} \in U^\k}\{\mr{E}(\langle \mb{g}^\k, \tilde{\mb{u}} -\tilde{\mb{p}}\rangle\vert \mc{F}^\k)\}]^2 \tag{Block Coordinate Structure}\\
&= \frac{1}{4q^2}\mr{E}\{\max_{\tilde{\mb{p}} \in V, \tilde{\mb{u}} \in U^\k}(\langle \nabla F(\mb{x}^\k), \tilde{\mb{u}} - \tilde{\mb{p}} \rangle)\}^2 \\
&\geq \frac{\Omega_{\mathcal{P}}^2}{4N^2 \kappa q^2} \mr{E}\{F(\mb{x}^\k) - F^*\}
\end{align*}

Now consider the four cases of step sizes in proof of Theorem 1 for iteration $(k)$ and block $l_i$, that is, 
\begin{description}
\item[($A^{(k)}_{[l_i]}$)]\hspace{0.5cm} $\bar{\gamma}^\k_{[l_i]}\geq 1$ and $\gamma^\k_{[l_i]} \leq 1$，
\item[($B^{(k)}_{[l_i]}$)]\hspace{0.5cm} $\bar{\gamma}^\k_{[l_i]} \geq 1$ and $\gamma^\k_{[l_i]} > 1$，
\item[($C^{(k)}_{[l_i]}$)]\hspace{0.5cm} $\bar{\gamma}^\k_{[l_i]} < 1$ and $\gamma^\k_{[l_i]} < \bar{\gamma}^\k_{[l_i]}$，
\item[($D^{(k)}_{[l_i]}$)]\hspace{0.5cm} $\bar{\gamma}^\k_{[l_i]} < 1$ and $\gamma^\k_{[l_i]} = \bar{\gamma}^\k_{[l_i]}$.
\end{description}
In case ($A^{(k)}_{[l_i]}$), let $\delta_{A^{(k)}_{[l_i]}}$ be the indicator function for this case, then we have
\begin{align*}
&\quad\quad \delta_{A^{(k)}_{[l_i]}} \{\gamma^{(k)}_{[l_i]}\langle \mb{U}_{[l_i]}^\top\mb{g}^{(k)}, \mathbf{d}^{(k)}_{[l_i]} \rangle + \frac{L}{2}(\gamma^{(k)}_{[l_i]})^2\norm{\mathbf{d}^{(k)}_{[l_i]}}^2 + \gamma^{(k)}_{[l_i]} \langle \nabla F(\mb{x}^\k) - \mb{g}^\k,  \mb{U}_{[l_i]}\mathbf{d}^{(k)}_{[l_i]} \rangle\} \\
&= \delta_{A^{(k)}_{[l_i]}} \{-\frac{\langle \mb{U}_{[l_i]}^\top\mb{g}^{(k)}, \mathbf{d}^{(k)}_{[l_i]} \rangle^2}{2L\norm{\mb{d}_{[l_i]}^\k}^2} + \frac{\langle \mb{U}_{[l_i]}^\top\mb{g}^{(k)}, \mathbf{d}^{(k)}_{[l_i]} \rangle}{L\norm{\mb{d}_{[l_i]}^\k}^2} \langle \nabla F(\mb{x}^\k) - \mb{g}^\k,  \mb{U}_{[l_i]}\mathbf{d}^{(k)}_{[l_i]} \rangle\}\\
&\leq \delta_{A^{(k)}_{[l_i]}} \{-\frac{\langle \mb{g}^{(k)}_{[l_i]}, \mathbf{d}^{(k)}_{[l_i]} \rangle^2}{2L\norm{\mb{d}_{[l_i]}^\k}^2} + \frac{G}{L}\norm{\mb{U}_{[l_i]}^\top(\nabla F(\mb{x}^\k) - \mb{g}^\k)} \} \\
&\leq \delta_{A^{(k)}_{[l_i]}} \{-\frac{\langle \mb{g}^{(k)}_{[l_i]}, \mathbf{d}^{(k)}_{[l_i]} \rangle^2}{2D \max (LD, G)} + \frac{G}{L}\norm{\mb{U}_{[l_i]}^\top(\nabla F(\mb{x}^\k) - \mb{g}^\k)} \}.
\end{align*}
In case ($B^{(k)}_{[l_i]}$), since $\gamma_{[l_i]}^\k \geq 1$ we have
\begin{align*}
&-\langle \mb{U}_{[l_i]}\mb{g}^\k, \mb{d}^\k_{[l_i]} \rangle > L\norm{\mb{d}_{[l_i]}^\k}^2 \quad\quad \text{and hence} \\
&\gamma_{[l_i]}^\k\langle\mb{U}_{[l_i]}\mb{g}^\k, \mb{d}^\k_{[l_i]}\rangle + \frac{L(\gamma_{[l_i]}^\k)^2}{2}\norm{\mb{d}^\k_{[l_i]}}^2 \leq  \langle\mb{U}_{[l_i]}\mb{g}^\k, \mb{d}^\k_{[l_i]}\rangle + \frac{L}{2}\norm{\mb{d}_{[l_i]}^\k}^2.
\end{align*}
Let $\delta_{B^{(k)}_{[l_i]}}$ be the indicator function for this case, then
\begin{align*}
&\quad\quad \delta_{B^{(k)}_{[l_i]}} \{\gamma^{(k)}_{[l_i]}\langle \mb{U}_{[l_i]}^\top\mb{g}^{(k)}, \mathbf{d}^{(k)}_{[l_i]} \rangle + \frac{L}{2}(\gamma^{(k)}_{[l_i]})^2\norm{\mathbf{d}^{(k)}_{[l_i]}}^2 + \gamma^{(k)}_{[l_i]} \langle \nabla F(\mb{x}^\k) - \mb{g}^\k,  \mb{U}_{[l_i]}\mathbf{d}^{(k)}_{[l_i]} \rangle\} \\
&\leq   \delta_{B^{(k)}_{[l_i]}} \{\langle \mb{U}_{[l_i]}^\top\mb{g}^{(k)}, \mathbf{d}^{(k)}_{[l_i]} \rangle + \frac{L}{2}\norm{\mathbf{d}^{(k)}_{[l_i]}}^2 + \gamma^{(k)}_{[l_i]} \langle \nabla F(\mb{x}^\k) - \mb{g}^\k,  \mb{U}_{[l_i]}\mathbf{d}^{(k)}_{[l_i]} \rangle \} \\
&\leq \delta_{B^{(k)}_{[l_i]}} \{\frac{1}{2}\langle \mb{U}_{[l_i]}^\top\mb{g}^{(k)}, \mathbf{d}^{(k)}_{[l_i]} \rangle + \frac{G}{L}\norm{\mb{U}_{[l_i]}^\top(\nabla  F(\mb{x}^\k) - \mb{g}^\k)} \} \\
&=  \delta_{B^{(k)}_{[l_i]}} \{- \frac{GD}{2} \frac{\langle -\mb{U}_{[l_i]}^\top\mb{g}^{(k)}, \mathbf{d}^{(k)}_{[l_i]} \rangle}{GD}+ \frac{G}{L}\norm{\mb{U}_{[l_i]}^\top (\nabla F(\mb{x}^\k) - \mb{g}^\k)}  \} \\
&\leq \delta_{B^{(k)}_{[l_i]}} \{-\frac{1}{2}\frac{\langle \mb{U}_{[l_i]}^\top\mb{g}^{(k)}, \mathbf{d}^{(k)}_{[l_i]} \rangle^2}{GD} +  \frac{G}{L}\norm{\mb{U}_{[l_i]}^\top(\nabla F(\mb{x}^\k) - \mb{g}^\k)}\} \\
&\leq \delta_{B^{(k)}_{[l_i]}} \{-\frac{\langle \mb{g}^{(k)}_{[l_i]}, \mathbf{d}^{(k)}_{[l_i]} \rangle^2}{2D\max (LD, G)} +  \frac{G}{L}\norm{\mb{U}_{[l_i]}^\top(\nabla F(\mb{x}^\k) - \mb{g}^\k)} \}
\end{align*}
In case ($C^{(k)}_{[l_i]}$), let $\delta_{C^{(k)}_{[l_i]}}$ denote the indicator function for this case and we can use exactly the same argument as in case ($A^{(k)}_{[l_i]}$) to obtain the following inequality
\begin{align*}
&\quad\quad \delta_{C^{(k)}_{[l_i]}} \{\gamma^{(k)}_{[l_i]}\langle \mb{U}_{[l_i]}^\top\mb{g}^{(k)}, \mathbf{d}^{(k)}_{[l_i]} \rangle + \frac{L}{2}(\gamma^{(k)}_{[l_i]})^2\norm{\mathbf{d}^{(k)}_{[l_i]}}^2 + \gamma^{(k)}_{[l_i]} \langle \nabla F(\mb{x}^\k) - \mb{g}^\k,  \mb{U}_{[l_i]}\mathbf{d}^{(k)}_{[l_i]} \rangle\} \\
&\leq \delta_{C^{(k)}_{[l_i]}} \{-\frac{\langle \mb{g}^{(k)}_{[l_i]}, \mathbf{d}^{(k)}_{[l_i]} \rangle^2}{2D\max (LD, G)} + \frac{G}{L}\norm{\mb{U}_{[l_i]}^\top(\nabla F(\mb{x}^\k) - \mb{g}^\k)} \}.
\end{align*}
In case ($D^{(k)}_{[l_i]}$), let $\delta_{D^{(k)}_{[l_i]}}$ denote the indicator function for this case and by the fact that $\gamma^\k_{[l_i]} = \bar{\gamma}^\k_{[l_i]} \leq -\langle \mb{g}^\k_{[l_i]}, \mb{d}^\k_{[l_i]} \rangle / (L \norm{\mb{d}_{[l_i]}}^2)$. Then,
\begin{align*}
&\quad\quad \delta_{D^{(k)}_{[l_i]}} \{\gamma^{(k)}_{[l_i]}\langle \mb{U}_{[l_i]}^\top\mb{g}^{(k)}, \mathbf{d}^{(k)}_{[l_i]} \rangle + \frac{L}{2}(\gamma^{(k)}_{[l_i]})^2\norm{\mathbf{d}^{(k)}_{[l_i]}}^2 + \gamma^{(k)}_{[l_i]} \langle \nabla F(\mb{x}^\k) - \mb{g}^\k,  \mb{U}_{[l_i]}\mathbf{d}^{(k)}_{[l_i]} \rangle\}\\
&\leq \delta_{D^{(k)}_{[l_i]}} \{\frac{\gamma^\k_{[l_i]}}{2}\langle \mb{g}^{(k)}_{[l_i]}, \mathbf{d}^{(k)}_{[l_i]} \rangle + \frac{G}{L}\norm{\mb{U}_{[l_i]}^\top(\nabla F(\mb{x}^\k) - \mb{g}^\k)} \} \\
&\leq \delta_{D^{(k)}_{[l_i]}} \{\frac{G}{L}\norm{\mb{U}_{[l_i]}^\top(\nabla F(\mb{x}^\k) - \mb{g}^\k)} \} \\
&= \delta_{D^{(k)}_{[l_i]}}\{-\frac{\langle \mb{g}^{(k)}_{[l_i]}, \mathbf{d}^{(k)}_{[l_i]} \rangle^2}{2L\norm{\mb{d}_{[l_i]}^\k}^2} + \frac{G}{L}\norm{\mb{U}_{[l_i]}^\top(\nabla F(\mb{x}^\k) - \mb{g}^\k)} \} + \delta_{D^{(k)}_{[l_i]}} \{\frac{\langle \mb{g}^{(k)}_{[l_i]}, \mathbf{d}^{(k)}_{[l_i]} \rangle^2}{2L\norm{\mb{d}_{[l_i]}^\k}^2}\}\\
&\leq \delta_{D^{(k)}_{[l_i]}} \{-\frac{\langle \mb{g}^{(k)}_{[l_i]}, \mathbf{d}^{(k)}_{[l_i]} \rangle^2}{2D\max (LD, G)} + \frac{G}{L}\norm{\mb{U}_{[l_i]}^\top(\nabla F(\mb{x}^\k) - \mb{g}^\k)}\} + \delta_{D^{(k)}_{[l_i]}} \{\frac{G^2}{2L}\}.
\end{align*}
Then combine the above four cases, we have
\begin{align*}
\{F(\mb{x}^{(k+1)}) - F^*\} \leq \{F(\mb{x}^{(k)}) - F^*\} + \sum_{i = 1}^r\{ -\frac{\langle \mb{g}^\k_{[l_i]}, \mb{d}^\k_{[l_i]} \rangle}{2D \max (G, LD)} + \frac{G}{L}\norm{\mb{U}_{[l_i]}^\top(\nabla F(\mb{x}^\k) - \mb{g}^\k)} \} + \frac{G^2}{2L}\sum_{i=1}^r\delta_{D^{(k)}_{[l_i]}}.
\end{align*}
Take expectation on both sides of the above inequality and use the assumption that sampled blocks and sampled gradient are independent,then we have
\begin{align*}
\mr{E}\{F(\mb{x}^{(k+1)}) - F^*\} &\leq \mr{E}\{F(\mb{x}^{(k)}) - F^*\} + \frac{r}{q}\sum_{i=1}^q\{ -\frac{\mr{E}\langle \mb{g}^\k_{[i]}, \mb{d}^\k_{[i]} \rangle^2}{2D \max (G, LD)} + \frac{G}{L}\mr{E}\norm{\mb{U}_{[i]}^\top(\nabla F(\mb{x}^\k) - \mb{g}^\k)} + \frac{G^2}{2L}\mr{P}(D_{[i]}^{(k)}) \} \\
&\leq \mr{E}\{F(\mb{x}^{(k)}) - F^*\} - \frac{r\mr{E}\langle \mb{g}^\k_{[i]}, \mb{d}^\k_{[i]} \rangle^2}{2D \max (G, LD)} + \frac{rG}{qL}\mr{E}\norm{\nabla F(\mb{x}^\k) - \mb{g}^\k} + \frac{rG^2}{2qL}\sum_{i=1}^q \mr{P}(D_{[i]}^{(k)}) \\
&\leq (1 - \frac{r\Omega_{\mathcal{P}}^2}{8N^2 \kappa q^2D \max (G, LD)} )\mr{E}\{F(\mb{x}^\k) - F^*\} +\frac{rG}{qL}\sqrt{(\frac{1}{m_k} - \frac{1}{n})C_2} + \frac{rG^2}{2qL}\sum_{i=1}^q \mr{P}(D_{[i]}^{(k)}).
\end{align*}

\noindent Let $\hat{\rho} = r\Omega_{\mathcal{P}}^2 /\{8N^2 \kappa q^2D \max (G, LD)\}$, use the assumption that $\max_{i}\{ \mr{P}(D_{[i]}^{(k)})\} < (1 - \hat{\rho})^{\lambda k}$ and $m_k \geq \lceil n / (1 + n(1 - \hat{\rho})^{2\mu k})\rceil$ . Then,
\begin{align*}
\mathbb{E}\{F(\mathbf{x}^{(k+1)}) - F^*\} &\leq (1 - \hat{\rho})\mr{E}\{F(\mb{x}^\k) - F^*\} + \frac{rG\sqrt{C_2}}{qL}(1 - \hat{\rho})^{\mu k} + \frac{rG^2}{2L} (1 - \hat{\rho})^{\lambda k} \\
&\leq (1 - \hat{\rho})^k\{F(\mb{x}^{(1)}) - F^*\} + \frac{rG\sqrt{C_2}}{qL}(1 - \hat{\rho})^{\mu k}\sum_{m=0}^{k-1}(1 - \hat{\rho})^{m(1 - \mu)} + \frac{rG^2}{2L}(1 - \hat{\rho})^{\lambda k}\sum_{m = 0}^{k-1}(1 - \hat{\rho})^{m(1 - \lambda)}\\
&\leq (1 - \hat{\rho})^k\{F(\mb{x}^{(1)}) - F^*\} + (\frac{rG\sqrt{C_2}}{qL}\frac{(1 - \hat{\rho})^{\mu k}}{1 - (1 - \hat{\rho})^{1 -\mu}} + \frac{rG^2}{2L}\frac{(1 - \hat{\rho})^{\lambda k}}{1 - (1 - \hat{\rho})^{1 - \lambda}} ) \\
&\leq (1 - \hat{\rho})^{\min (\lambda, \mu) k}\{F(\mb{x}^{(1)}) - F^*  + \frac{rG\sqrt{C_2}}{qL\{1 - (1 - \hat{\rho})^{1 -\mu}\}} + \frac{rG^2}{2L\{1 - (1 - \hat{\rho})^{1 - \lambda}\}}\} \\
&= (1 - \hat{\rho})^{\min (\mu, \lambda) k} C_4
\end{align*}
\end{proof}
\bibliography{icml_fw}

\begin{thebibliography}{20}
\providecommand{\natexlab}[1]{#1}
\providecommand{\url}[1]{\texttt{#1}}
\expandafter\ifx\csname urlstyle\endcsname\relax
  \providecommand{\doi}[1]{doi: #1}\else
  \providecommand{\doi}{doi: \begingroup \urlstyle{rm}\Url}\fi

\bibitem[Abadie(1970)]{Wol70}
Abadie, J. (ed.).
\newblock \emph{Integer and Nonlinear Programming}.
\newblock North-Holland, Amsterdam, 1970.

\bibitem[Beck \& Shtern(2015)Beck and Shtern]{BS15}
Beck, A. and Shtern, S.
\newblock Linearly convergent away-step conditional gradient for non-strongly
  convex functions.
\newblock \emph{arXiv: 1504.05002}, 2015.

\bibitem[Chen et~al.(2011)Chen, Lin, Kim, Carbonell, and Xing]{CLKCX11}
Chen, X., Lin, Q., Kim, S., Carbonell, J.~G., and Xing, E.~P.
\newblock Smoothing proximal gradient method for general structured sparse
  learning.
\newblock In \emph{Proceedings of the Conference on Uncertainty in Artificial
  Intelligence (UAI)}, pp.\  105–114, 2011.

\bibitem[Frank \& Wolfe(1956)Frank and Wolfe]{FW56}
Frank, M. and Wolfe, P.
\newblock An algorithm for quadratic programming.
\newblock \emph{Naval Research Logistics Quarterly}, 3:\penalty0 95–110,
  1956.

\bibitem[Garber \& Hazan(2013)Garber and Hazan]{GH13}
Garber, D. and Hazan, E.
\newblock A linearly convergent conditional gradient algorithm with
  applications to online and stochastic optimization.
\newblock \emph{arXiv:1301.4666}, 2013.

\bibitem[Garber \& Hazan(2015)Garber and Hazan]{GH15}
Garber, D. and Hazan, E.
\newblock Faster rates for the {F}rank-{W}olfe method over strongly-convex
  sets.
\newblock In \emph{Proceedings of the 32nd International Conference on Machine
  Learning (ICML 2015)}, pp.\  541–549, 2015.

\bibitem[Guelat \& Marcotte(1986)Guelat and Marcotte]{GM86}
Guelat, J. and Marcotte, P.
\newblock Some comments on {W}olfe's `away step'.
\newblock \emph{Mathematical Programming}, 35\penalty0 (1):\penalty0 110–119,
  1986.

\bibitem[Hoffman(1952)]{Hof52}
Hoffman, A.
\newblock On approximate solutions of systems of linear inequalities.
\newblock \emph{Journal of Research of the National Bureau of Standards},
  49\penalty0 (4):\penalty0 263–265, 1952.

\bibitem[Johnson \& Zhang(2013)Johnson and Zhang]{JZ13}
Johnson, Rie and Zhang, Tong.
\newblock Accelerating stochastic gradient descent using predictive variance
  reduction.
\newblock In Burges, C.J.C., Bottou, L., Welling, M., Ghahramani, Z., and
  Weinberger, K.Q. (eds.), \emph{Advances in Neural Information Processing
  Systems 26}, pp.\  315--323, 2013.

\bibitem[Kim et~al.(2009)Kim, Sohn, and Xing]{KSX09}
Kim, S., Sohn, K.~A., and Xing, E.~P.
\newblock A multivariate regression approach to association analysis of a
  quantitative trait network.
\newblock \emph{Bioinformatics}, 25:\penalty0 204--212, 2009.

\bibitem[Lacoste-Julien \& Jaggi(2014)Lacoste-Julien and Jaggi]{LJJ13}
Lacoste-Julien, S. and Jaggi, M.
\newblock An affine invariant linear convergence analysis for {F}rank-{W}olfe
  algorithms.
\newblock \emph{arXiv:1312.7864v2}, 2014.

\bibitem[Lacoste-Julien et~al.(2013)Lacoste-Julien, Jaggi, Schmidt, and
  Pletscher]{LJJSP13}
Lacoste-Julien, S., Jaggi, M., Schmidt, M., and Pletscher, P.
\newblock Block-coordinate frank-wolfe optimization for structural svms.
\newblock \emph{In Proceedings of the 30th International Conference on Machine
  Learning (ICML-13)}, 28:\penalty0 53–61, 2013.

\bibitem[Lafond et~al.(2015)Lafond, Wai, and Moulines]{LWM15}
Lafond, J., Wai, H., and Moulines, E.
\newblock Convergence analysis of a stochastic projection-free algorithm.
\newblock \emph{arXiv:1510.01171}, 2015.

\bibitem[Lan(2013)]{LAN13}
Lan, G.
\newblock The complexity of large-scale convex programming under a linear
  optimization oracle.
\newblock \emph{arXiv:1309.5550}, 2013.

\bibitem[Levitin \& Polyak(1966)Levitin and Polyak]{LP66}
Levitin, E. and Polyak, B.~T.
\newblock Constrained minimization methods.
\newblock \emph{USSR Computational Mathematics and Mathematical Physics},
  6\penalty0 (5):\penalty0 787–823, 1966.

\bibitem[Mu et~al.(2015)Mu, Zhang, Wright, and Goldfarb]{MZWG15}
Mu, C., Zhang, Y., Wright, J., and Goldfarb, D.
\newblock Scalable robust matrix recovery: {F}rank-{W}olfe meets proximal
  methods.
\newblock \emph{arXiv:1403.7588}, 2015.

\bibitem[Nesterov(2015)]{N15}
Nesterov, Yu.
\newblock Complexity bounds for primal-dual methods minimizing the model of
  objective function.
\newblock Technical report, Universit\'e catholique de Louvain, Center for
  Operations Research and Econometrics ´ (CORE), 2015.

\bibitem[Ouyang \& Gray(2010)Ouyang and Gray]{OG10}
Ouyang, H. and Gray, A.
\newblock Fast stochastic {F}rank-{W}olfe algorithms for nonlinear {SVM}s.
\newblock In \emph{SDM}, pp.\  245–256, 2010.

\bibitem[Taskar et~al.(2003)Taskar, Guestrin, and Koller]{TGK03}
Taskar, B., Guestrin, C., and Koller, D.
\newblock Max-margin {M}arkov networks.
\newblock In \emph{Advances in Neural Information Processing Systems (NIPS)},
  pp.\  25--32, 2003.

\bibitem[Wang \& Lin(2014)Wang and Lin]{WL14}
Wang, Po-Wei and Lin, Chih-Jen.
\newblock Iteration complexity of feasible descent methods for convex
  optimization.
\newblock \emph{The Journal of Machine Learning Research}, 15\penalty0
  (1):\penalty0 1523--1548, 2014.

\end{thebibliography}
\bibliographystyle{icml2016}
\end{document}